\documentclass[12pt]{amsart}
\usepackage[all]{xy}
\usepackage{amssymb}
\usepackage{amsmath}
\usepackage{mathrsfs}
\usepackage{amsfonts}
\usepackage{amsmath,amsfonts,amssymb}
\usepackage{cases}
\usepackage{indentfirst}
\usepackage{ifthen}
\usepackage{amsthm}
\usepackage{graphicx}
\usepackage{epstopdf}
\usepackage{changepage}
\usepackage{geometry}
\usepackage{bbm}
\usepackage{color}
\usepackage[normalem]{ulem}
\usepackage{mathdots}
\usepackage{stmaryrd}
\usepackage{hyperref}
\usepackage{tikz-cd} 
\usepackage[framemethod=tikz]{mdframed}
\usepackage{lipsum}
\usepackage{multirow}
\usepackage{multicol}
\usepackage{bm}
\numberwithin{equation}{subsection}

\newcommand{\RR}{\mathbb R}
\newcommand{\CC}{\mathbb C}

\newcommand{\ZZ}{\mathbb Z}

\newcommand{\bP}{\mathbf P}
\newcommand{\bT}{\mathbf T}
\newcommand{\bB}{\mathbf B}
\newcommand{\bH}{\mathbf H}
\newcommand{\bG}{\mathbf G}

\newcommand{\bM}{\mathbf M}
\newcommand{\bN}{\mathbf N}

\newcommand{\bU}{\mathbf U}

\newcommand{\bZ}{\mathbf Z}

\newcommand{\cF}{\mathcal F}

\newcommand{\cH}{\mathcal H}

\newcommand{\cO}{\mathcal O}

\newcommand{\cS}{\mathcal S}

\newcommand{\fp}{\mathfrak p}

\newcommand{\SL}{\mathrm{SL}}

\newcommand{\PGL}{\mathrm{PGL}}
\newcommand{\GL}{\mathrm{GL}}

\newcommand{\Spin}{\mathrm{Spin}}

\newcommand{\rG}{\mathrm{G}}

\newcommand{\rE}{\mathrm{E}}

\newcommand{\GG}{\mathbb G}
\newcommand{\PP}{\mathbb P}

\newcommand{\diag}{\mathrm{diag}}

\newcommand{\Hom}{\operatorname{Hom}}
\newcommand{\Stab}{\operatorname{Stab}}

\newcommand{\End}{\operatorname{End}}
\newcommand{\der}{\mathrm{der}}

\newcommand{\Irr}{\mathrm{Irr}}

\newcommand{\ind}{\mathrm{ind}}

\newtheorem{theorem}{Theorem}[section]

\newtheorem{proposition}{Proposition}[section]
\newtheorem{lemma}{Lemma}[section]
\newtheorem{corollary}{Corollary}[section]

\newtheorem*{remark}{Remark}
\theoremstyle{remark}

\usepackage[english]{babel}

\hypersetup{
  colorlinks=true,
  linkcolor=blue,
  citecolor=blue,
  urlcolor=blue,
  linktoc=all,
  pdftitle={A variation on local trilinear form},
  pdfauthor={Yiyang Wang}
}

\begin{document}

\title{A variation on local trilinear form}
\author{Yiyang Wang}
\address{Department of Mathematics, Kyoto University, Kitashirakawa Oiwake-cho, Sakyo-ku, Kyoto 606-8502, Japan}
\email{wang.yiyang.67u@st.kyoto-u.ac.jp}

\begin{abstract}
\mathversion{normal}
Let $F$ be a $p$-adic field and $M$ the subgroup of $\GL_2^3(F)$ consisting of triples with equal determinants. For each unitary principal series $L$-packet of $M$ whose central character is trivial on the diagonal $F^\times$, we determine its unique member distinguished by the diagonal $\GL_2(F)$.
\end{abstract}

\maketitle
\addtocontents{toc}{\setcounter{tocdepth}{1}} %only sections (no subsections) in the contents
\tableofcontents

\section{Introduction}\label{sec:introduction}
% ===== Begin 01_intro.tex =====
\subsection{Motivation}\label{subsec:motivation}
Let $F$ be a $p$-adic field and $\bG=\GL_2$ over $F$. We set $\bM^{\#}:=\bG\times \bG\times \bG =\GL_2^3$,
\[
\bM=(\GL_2^3)^{\det}:=\{(g_1,g_2,g_3)\in \GL_2^3\mid \det(g_1)=\det(g_2)=\det(g_3)\}\subset \bM^{\#},
\]
and $\bG^\Delta$ the diagonal subgroup of $\bM$. Let $\bm{\pi}^{\#}\in \Irr(M^{\#})$ be an irreducible smooth representation of $M^{\#}=\bM^{\#}(F)$. From the work of Tadi\'c \cite{tadic}, we know that the restriction $\bm{\pi}:=\bm{\pi}^{\#}|_{M}$ to $M =\bM(F)$ admits a semisimple decomposition $$ \bm{\pi}= \bm{\sigma}_1\oplus \cdots \oplus \bm{\sigma}_n,\qquad \bm{\sigma}_i\in \Irr(M)$$ which is of finite length and multiplicity free. Moreover, each irreducible representation of $M$ can be embedded into such a restriction and the set of constituents $\{\bm{\sigma}_1,\cdots,\bm{\sigma}_n\}$ of $\bm{\pi}=\bm{\pi}^{\#}|_M$ forms an L-packet of $M$ (see \S~\ref{sec:representations} for details). Suppose the central character of $\bm{\pi}^{\#}$ is further trivial on the diagonal center $Z^\Delta\simeq F^\times \subset G^\Delta$, then by the multiplicity one theorem of D. Prasad \cite[Theorem 1.1]{prasad}, we have 
\[
\dim \Hom_{G^\Delta}(\bm{\pi}^{\#}, \mathbf{1})\leq 1.
\]
Since $G^\Delta$ is also a subgroup of $M$, we have 
\[
\Hom_{G^\Delta}(\bm{\pi}, \mathbf{1}) = \bigoplus_{i=1}^n \Hom_{G^\Delta}(\bm{\sigma}_i,\mathbf{1}),
\]
which leads to the following natural problem.
\begin{itemize}
    \item \textbf{Problem 1}. When $\dim\Hom_{G^\Delta}(\bm{\pi}^{\#},\mathbf{1}) = 1$, which $\Hom_{G^\Delta}(\bm{\sigma}_i,\mathbf{1})$ is nonzero?
\end{itemize}

The purpose of this paper is to deal with this question when $\bm{\pi}^{\#} \in \Irr(M^{\#})$ is an irreducible principal series induced from a unitary character. Before explaining the results, let us briefly explain our motivation for this question.

Let $\bG_E=\Spin_8$ be the connected split simple group over $F$ that is simply-connected and of type $\mathrm{D}_4$. Up to conjugacy, $\bG_E$ has a unique Heisenberg parabolic $\bP_E = \bM_E\bN_E$, whose Levi $\bM_E$ is isomorphic to $\bM$ as an $F$-group. We fix an isomorphism $\bM_E\simeq \bM = (\GL_2^3)^{\det}$ such that the modulus character $\delta_{P_E}|_{M_E}$ is identified with $|\det|^5: M \longrightarrow \RR^{\times}_{>0}$, where $$\det: M \ni (g_1,g_2,g_3)\longmapsto \det(g_1)\in F^\times. $$
On the other hand, the split exceptional group $\bH =\rG_2$ is naturally embedded into $\bG_E$ via the triality realization, of which the detail is not needed here. There is a dual pair $\SL_2^3/\mu_2^\Delta\times \bG_E$ in the adjoint split $\rE_7$, where the $F$-group $\SL_2^3/\mu_2^\Delta$ is just isomorphic to $\bM/\GG_m^\Delta$. In the study of the theta correspondence for this dual pair, we are naturally led to the following distinction problem.
\begin{itemize}
    \item \textbf{Problem 2}. Let $\bm{\sigma}\in \Irr(M_E)$ be an irreducible tempered representation of $M_E$ whose central character $\omega_{\bm{\sigma}}$ is trivial on the diagonal $F^\times$. Given $s>0$, we denote by $I_{P_E}(\bm{\sigma},s):=I_{P_E}^{G_E}(\bm{\sigma}\otimes |\det|^{s})$ the corresponding standard module, and $J_{P_E}(\bm{\sigma},s)$ its Langlands quotient. Determine
    \[
    \dim \Hom_{H}(I_{P_E}(\bm{\sigma},s), \mathbf{1})\qquad \text{and}\qquad \dim \Hom_{H}(J_{P_E}(\bm{\sigma},s), \mathbf{1}).
    \]
\end{itemize}
The standard approach to deal with $\Hom_{H}(I_{P_E}(\bm{\sigma},s), \mathbf{1})$ is the Mackey theory (see \cite[\S 1]{blanc-delorme}). It turns out that there are five double cosets in $H\backslash G_E/P_E =\rG_2(F)\backslash G_E/P_E$, and we take a set of representatives $\{g_i\}_{i=1}^5$. For $1\leq i\leq 5$, we set $P_i:= H\cap (g_i P_E g_i^{-1})$ and denote by $^{g_i}\bm{\sigma}_s$ the representation of $g_i P_E g_i^{-1}$ given by
\[
^{g_i}\bm{\sigma}_s(h):=(\bm{\sigma}\otimes |\det|^s)(g_i^{-1}hg_i),\qquad h\in P_i.
\]
Then the Mackey filtration of $I_{P_E}(\bm{\sigma},s)|_{H}$ has constituents 
\[
X_i(s):=\mathrm{ind}_{P_i}^{H}\left({}^{g_i}(\bm{\sigma}_s\delta_{P_E}^{1/2})|_{P_i}\right),\qquad 1\leq i\leq 5,
\]
where $\mathrm{ind}(-)$ denotes the unnormalized compact induction. 

The unique closed orbit among the five orbits is $H 1 P_E$, and we take $g_1 = 1$ so that $P_1 = P_E\cap H = M_1 N_1$ is the Heisenberg parabolic of $H=\rG_2(F)$. The pair $(\bM_E,\bM_1)$ is isomorphic to $(\bM,\bG^\Delta)\simeq ((\GL_2^3)^{\det},\GL_2^\Delta)$, and with the computational facts
\[
\delta_{P_E} = |\det|^5,\qquad \delta_{P_E}|_{M_1} = \delta_{P_1}^{5/3},
\]
we can then get 
\begin{equation}
    \label{eq:closed-orbit-hom}\Hom_{H}(X_1(s),\mathbf{1})=\Hom_{M_1}\left( (\bm{\sigma} \delta_{P_1}^{\frac{2s-1}{6}}),\mathbf{1}\right)\simeq \Hom_{M_1}(\bm{\sigma},\delta_{P_1}^{\frac{1-2s}{6}}).
\end{equation}
Based on an analysis of the orbits and exponents, one can in fact verify that
\begin{equation}
    \label{eq:non-closed-orbit-vanishing} \Hom_{H}(X_i(s),\mathbf{1}),~\mathrm{Ext}^1_{H}(X_i(s),\mathbf{1})=0,\qquad 2\leq i\leq 5
\end{equation} at least for $s>0$ belonging to a dense open set $U\subset (0,+\infty)$ including $s=1/2$. We would then get 
\[
\Hom_H(I_{P_E}(\bm{\sigma},s),\mathbf{1})\simeq\begin{cases}
    0,& s\in U-\{1/2\},\\
    \Hom_{M_1}(\bm{\sigma},\mathbf{1}),& s= 1/2,
\end{cases}
\]
 
since we assume that $\omega_{\bm{\sigma}}$ is trivial on the diagonal $F^\times$. Thus the study of Problem 2 can be ``generically" reduced to Problem 1.

\begin{remark}
    The analysis of the orbits and  the vanishing result \eqref{eq:non-closed-orbit-vanishing} will be included in a subsequent paper. We note that the classification of the five orbits of $\rG_2(F)\backslash G_E/P_E$ was given in the pioneering work \cite{jiang}, and also explained in \cite{pollack-wan-zydor}.
\end{remark}

\subsection{Main results}\label{subsec:main-results}
Let $\bB=\bT\bU$ be the standard upper triangular Borel of $\bG \simeq \GL_2$, and $\bB_M^{\#} = \bT_M^{\#} \bU_M^{\#}$ (resp. $\bB_M = \bT_M \bU_M$) the standard upper triangular Borel of $\bM^{\#}$ (resp. $\bM$). Given characters $\mu,\nu: F^\times \longrightarrow \CC^\times$, we denote by $\mu \times \nu$ the normalized induction $I_B^G(\mu\boxtimes \nu)$. For $1\leq i\leq 3$, let $\mu_i ,\nu_i: F^\times \longrightarrow \CC^1$ be unitary characters, and we set 
\[
\bm{\pi}^{\#}=\boxtimes_{i=1}^3(\mu_i\times\nu_i),\qquad
\bm{\pi}=\bm{\pi}^{\#}|_M.
\]
In this context we denote the inducing characters from $T_M^{\#}$ and $T_M$ by 
\[ \bm{\chi}^{\#} = (\mu_1,\nu_1;\mu_2,\nu_2;\mu_3,\nu_3),\qquad \bm{\chi}:=\bm{\chi}^{\#}|_{T_M}.\]
We assume further $\prod_{i=1}^3\mu_i\nu_i=\mathbf{1}$, i.e. the central character $\omega_{\bm{\pi}^{\#}}$ is trivial on the diagonal center, and write $\xi_i=\mu_i\nu_i^{-1}$. 

Let $R_{\bm{\chi}}$ be the R-group for $\bm{\chi}$, and we fix the standard Whittaker datum $\mathfrak{w}_M$ of $M$ (see \S~\ref{sec:representations} for detail). Then the theory of R-groups gives a labeling
\[
\bm{\pi}\simeq\bigoplus_{\kappa\in\Irr(R_{\bm{\chi}})}\bm{\sigma}_\kappa,
\]
with $\bm{\sigma}_{\mathbf{1}}$ the unique $\mathfrak{w}_M$-generic member. It turns out that the order of $R_{\bm{\chi}}$ can only be $4,2,1$, and we can easily classify the cases when $|R_{\bm{\chi}}|>1$, as follows.
\begin{enumerate}
    \item[(a)] If $|R_{\bm{\chi}}|=4$, then $\xi_1 = \xi_2 = \xi_3$ is a nontrivial quadratic character. (We call this the case of \textbf{order 4}.)
    \item[(b)] If $|R_{\bm{\chi}}|=2$, then exactly one of the following holds.
    \begin{enumerate}
        \item[(1)] Up to permutation $\xi_1 = \xi_2 = \xi$ is a nontrivial quadratic character, and $\xi_3\neq \xi$. (We call this the case of \textbf{order 2, type I}.)
        \item[(2)] The three characters $\xi_1,\xi_2,\xi_3$ are mutually distinct nontrivial quadratic characters, and $\xi_1\xi_2\xi_3 = \mathbf{1}$. (We call this the case of \textbf{order 2, type II}.)
    \end{enumerate}
\end{enumerate}
When $|R_{\bm{\chi}}|=2$, let us denote by $\bm{\sigma}_+$ the unique $\mathfrak{w}_M$-generic member, and $\bm{\sigma}_-$ the other member. (The notation $\bm{\sigma}_{\pm}$ here is identified with $\bm{\sigma}_{\pm 1}$ whenever needed.) Now we can state our results as follows.

\begin{theorem} \label{theorem:main-theorem}
    Notations as above.
    \begin{enumerate}
        \item In the case of order 4, the $G^{\Delta}$-distinguished member is $\bm{\sigma}_{\mathbf{1}}$, i.e. the $\mathfrak{w}_M$-generic one.
        \item In the case of order 2, type I, say $\xi_1 = \xi_2 = \xi$ and $\xi_3\neq \xi$ up to permutation, we put $s:=\xi(-1)$. Then the $G^{\Delta}$-distinguished member is $\bm{\sigma}_s$.
        \item In the case of order 2, type II, we put $Y=\{\mathbf{1},\xi_1,\xi_2,\xi_3\}$ and 
        \[ s=\prod_{\eta\in Y}\varepsilon(\frac12, \eta,\psi)\in\{\pm1\}.\]
Then the $G^{\Delta}$-distinguished member is $\bm{\sigma}_s$.
    \end{enumerate}
\end{theorem}

First we mention that if all $\mu_i,\nu_i$ are unramified and $|R_{\bm{\chi}}|>1$, then we are either in the case of order 4, or in the case of order 2 and type I, with $\xi(-1) = 1$. We set $K_M:= \GL_2^3(\cO_F)\cap M$ (where $\cO_F$ is the ring of integers of $F$), then we can easily check that the $\mathfrak{w}_M$-generic member is really $K_M$-unramified (see \S~\ref{subsec:supplements}). Thus in this unramified setting the $G^{\Delta}$-distinguished member is exactly the ($K_M$-)unramified one, and this is compatible with D. Prasad's theorem in the unramified case. 

Let us briefly explain how to prove Theorem \ref{theorem:main-theorem}. Let the symmetric group $S_3$ act on $\bM^{\#},\bM$ in the usual way by permuting the three factors. In the case of order 4, it turns out that the packet $\{\bm{\sigma}_{\kappa}\}$ is $S_3$-stable, while the $S_3$-action also preserves the Hom-space $\Hom_{G^\Delta}(\bm{\pi},\mathbf{1})$. Then the $G^{\Delta}$-distinguished member in $\{\bm{\sigma}_{\kappa}\}$ must be fixed by $S_3$, which is just $\bm{\sigma}_{\mathbf{1}}$. The two cases of order 2 are more interesting. In such situations the theory of R-groups gives us a normalized intertwining operator $R(w,\bm{\chi})\in \End_{M}(\bm{\pi})$, which acts on $\bm{\sigma}_{\pm}$ by $\pm 1$. Then we take a nonzero trilinear form $L\in \Hom_{G^{\Delta}}(\bm{\pi},\mathbf{1})$ and determine the constant $c \in \{ \pm 1\}$ such that 
\[ L \circ R(w,\bm{\chi}) = c L.\]
Finally, the $G^{\Delta}$-distinguished constituent is $\bm{\sigma}_c$.

The evaluation of the constant $c$ is a pleasant computational problem related to trilinear forms. In the type I case, we use D. Prasad's Mackey filtration  together with an explicit formula for the torus period to compute it (see \S~\ref{subsec:open-orbit-flip}).  Somewhat surprising to the author, the constant $c$ in the type II case has essentially been computed by Venkatesh--Wang in another context in \cite[Lemma 11.4.2]{venkatesh-wang} (while their motivation has nothing to do with our problem). Their method is purely analytic, which expresses the corresponding intertwining operator $R(w,\bm{\chi})$ as Fourier transform on a specific model for principal series. The reader may compare our method for type I case with this beautiful analytic approach. In \S~\ref{sec:type-ii} we match the result of \cite[Lemma 11.4.2]{venkatesh-wang} with our setting, and the technicality is mainly to extend their formulae for three $\PGL_2$ to $\bM$ and the switch between several different models.

We close this subsection with a brief outline of this paper.
\begin{itemize}
    \item In \S~\ref{subsec:notation}, we list some general notations used throughout the paper.
    \item In \S~\ref{sec:representations}, we deal with the representation theory of the group $M$, focusing on the restriction of irreducible principal series, and the R-group.
    \item In \S~\ref{sec:results}, we restate the main results with some remarks, and also prove the order 4 case. At the end of this section we add some additional results; they are not quite related to the main Theorem, but will be used in a subsequent paper. 
    \item In \S~\ref{sec:type-i} (resp. \S~\ref{sec:type-ii}), we complete the proof of the case of order 2, type I (resp. type II).
\end{itemize}

\subsection*{Acknowledgements} 

The author would like to thank his advisor A. Ichino for his constant support. This work is supported by JSPS KAKENHI Grant Number JP25KJ1457.

\subsection{Some general notations}\label{subsec:notation}
Here are some notations for representation theory of $p$-adic groups.
\begin{itemize}
\item We fix a $p$-adic field $F$, with $\mathcal O_F$ its ring of integers, and $\mathfrak p_F\subset \mathcal O_F$ the maximal ideal with a fixed uniformizer $\varpi_F$. The absolute value $|\cdot|$ on $F$ is normalized by $|\varpi_F|=q_F^{-1}$, where $q_F=|\cO_F/\fp_F|$.
\item  We fix a nontrivial unitary additive character $\psi:F\longrightarrow \CC^1$. We then denote by $dx$ the additive Haar measure on $F$ that is self-dual with respect to $\psi$, and set $d^\times x = dx/|x|$, which is a multiplicative Haar measure on $F^\times$. 

We usually use the symbol $\mathbf{1}$ to denote the trivial character.
\item An algebraic group over $F$ will be denoted by a boldfaced letter (such as $\bG,\bH,\bM$), and their groups of $F$-points will be denoted by the corresponding ``usual" letter (such as $G,H,M$).
\item All representations in this paper are complex smooth representations. For a $p$-adic group $H$, we denote by $\Irr(H)$ the set of (isomorphism classes of) irreducible smooth representations of $H$, and by
\[
X(H):=\Hom_{\mathrm{sm}}(H,\CC^\times)
\]
the group of its smooth characters. If $H_1$ is a closed normal subgroup of $H$, we shall identify 
\[X(H/H_1) = \{\eta \in X(H)\mid \eta|_{H_1} = \mathbf{1}\}\]
as a subgroup of $X(H)$.

\item For a representation $(\tau,V_\tau)$, we denote its contragredient by $(\tau^\vee,V_{\tau}^\vee)$, and its central character by $\omega_\tau$ whenever it does exist. If $\tau$ is a representation of $\GL_2(F)$ and $\eta:F^\times\longrightarrow \CC^\times$ a character, we use the simplified notation $\tau\otimes \eta:= \tau\otimes (\eta \circ \det)$ for the twist by $\eta$. 

For a representation $\tau_i$ of a group $H_i$ ($i=1,2$), we denote by $\tau_1\boxtimes \tau_2$ the external tensor product representation of $H_1\times H_2$.
\item For a $p$-adic group $H$ with a closed subgroup $P$ and a representation $(\tau,V_{\tau})$ of $P$, we denote by $\ind_P^H(\tau)$ the unnormalized compact induction, whose underlying space $\ind_P^H(V_{\tau})$ is the set of smooth functions $f:H\to V_\tau$ with $f(ph)=\tau(p)f(h)$ for $p\in P, h\in H$, and the support of $f$ is compact modulo $P$. Then the normalized compact induction will be denoted by $$I_P^H(\tau):=\mathrm{ind}_P^H(\delta_P^{1/2}\tau),$$ where $\delta_P: P\longrightarrow \RR_{>0}^\times$ is the modulus character of $P$.
\end{itemize}
Whenever convenient, we shall not distinguish a representation from its underlying space. For example, it is usually convenient to write $\Hom_H(\tau,\mathbf{1})$ instead of $\Hom_H(V_\tau,\CC)$, or $f\in I_P^H(\tau)$ instead of $f\in I_P^H(V_\tau)$. But we remark that this is not always the case, because we will have to switch between different models in many contexts.

We next fix some notations for groups and their elements or characters.
\begin{itemize} 
\item We fix $\bG=\GL_2$, with the standard upper triangular Borel subgroup $\bB=\bT\bU$. We use the notations
\[
u(x)=\begin{bmatrix}1&x\\0&1\end{bmatrix},\qquad
\overline u(x)=\begin{bmatrix}1&0\\x&1\end{bmatrix},\qquad
\dot{w}_0=\begin{bmatrix}0&1\\-1&0\end{bmatrix},\qquad s_0=\begin{bmatrix}0&1\\1&0\end{bmatrix}.
\]
\item We set
\[
\bM^{\#}=\bG^3,\qquad
\bM=(\GL_2^3)^{\det}:=\{(g_1,g_2,g_3)\in\GL_2^3:\det(g_1)=\det(g_2)=\det(g_3)\}.
\]
The upper triangular Borel subgroups of $\bM^{\#}$ and $\bM$ are denoted by
\[
\bB_M^{\#}=\bT_M^{\#}\bU_M=\bB^3,
\qquad
\bB_M=\bT_M\bU_M=\bB_M^{\#}\cap\bM,
\]
 The diagonal copy of $\bG$ in $\bM$ and its center are denoted by
\[
\bG^\Delta=\{(g,g,g):g\in\bG\},\qquad
\bZ^\Delta=\{(zI_2,zI_2,zI_2):z\in\GG_m\}.
\]

\end{itemize}

Finally, several words for notations of parabolic induction. As we have mentioned, in many contexts we want to switch between different models of the same induced representation. Given characters $\mu,\nu: F^\times \longrightarrow \CC^\times$, we denote by $(I_{\mu,\nu},V_{\mu,\nu})$ the ``usual" induced model $I_B^G(\mu \boxtimes \nu)$, and we use $\mu \times \nu$ to indicate the isomorphism class of $I_B^G(\mu \boxtimes \nu)$. Given a character $\bm\chi^{\#}$ (resp. $\bm\chi$) of $T_M^{\#}$ (resp. $T_M$), we write 
\[
I(\bm\chi^{\#}):=I_{B_M^{\#}}^{M^{\#}}(\bm\chi^{\#}),
\qquad
I(\bm\chi):=I_{B_M}^{M}(\bm\chi)
\]
for the corresponding induced models. If $\bm{\chi}^{\#} = (\mu_1,\nu_1;\mu_2,\nu_2;\mu_3,\nu_3)$, then $I(\bm\chi^{\#})= \boxtimes_{i=1}^3 I_{\mu_i , \nu_i}$ with its underlying space $\otimes_{i=1}^3 V_{\mu_i,\nu_i}$.
% ===== End 01_intro.tex =====

\section{Representations of $M$}\label{sec:representations}
% ===== Begin 02_repn of M.tex =====

The purpose of this subsection is to deal with irreducible representations of $M$ and their L-parameters. Following \S~\ref{subsec:notation} we write $\bM^{\#} = \GL_2^3$, $\bM = (\GL_2^3)^{\det}$, and there is an exact sequence 
\[
1\longrightarrow \bM \longrightarrow \bM^{\#} \overset{\delta}{\longrightarrow} \mathbb{G}_m^2\longrightarrow 1
\]
of algebraic groups, with 
\[
\delta:\GL_2^3\longrightarrow \GG_m^2,\qquad (g_1,g_2,g_3)\mapsto (\det (g_1 g_2^{-1}),\det(g_2 g_3^{-1})).
\]
Note that on the level of $F$-points we also have the exact sequence 
\[
1\longrightarrow M \longrightarrow M^{\#} \overset{\delta }{\longrightarrow} (F^\times)^2\longrightarrow 1,
\]
so that we can identify $(\bM^{\#}/\bM)(F)= M^{\#}/M$.

On the dual side we have an exact sequence
\[
1\longrightarrow Z'\longrightarrow \widehat{M^{\#}} \longrightarrow \widehat{M}\longrightarrow 1,
\]
where after identifying $\widehat{M^{\#}}=\GL_2^3(\CC)$ the central torus $Z'$ can be written as
\[
Z' = \{(z_1 I_2, z_2 I_2, z_3 I_2) \mid z_k \in \CC^\times \text{ and }z_1 z_2 z_3 = 1\},
\]
and $\widehat{M}\simeq \widehat{M^{\#}}/Z'$ as complex groups. We regard $Z'$ as the dual group of the $F$-group $\bM^{\#}/\bM\simeq \GG_m^2$, and then by local class field theory we have an isomorphism 
\[
\Hom(L_F,Z')\simeq
X(M^{\#}/M):=\{\chi \in X(M^{\#}) \mid \chi|_{M}=\mathbf{1}\}.
\]

We start from a brief review of some classical results of Tadi\'c \cite[\S 2]{tadic}. 

\subsection{Review of Tadi\'c's results}\label{subsec:tadic}

Let $\bH$ temporarily be a general connected reductive group over $F$, with derived subgroup $\bH^{\der}$. Suppose $\bH_1$ is an algebraic subgroup of $\bH$ containing $\bH^{\der}$, and in particular $H_1=\bH_1(F)$ is a normal closed subgroup of $H=\bH(F)$ with abelian quotient. The group $H$ acts on $\Irr(H_1)$ in a natural way (say from the right): for $x\in H$, $(\sigma, V) \in \Irr(H_1)$, the action gives $(\sigma^x, V) \in \Irr(H_1)$ defined by
\[
\sigma^x(g):=\sigma(x g x^{-1}),\quad g\in H_1.
\]
This action factors through $H/H_1Z(H)$, which is a finite abelian group. The work of Tadi\'c provides us with the following results.
\begin{enumerate}
\item For any $\pi \in \Irr(H)$, the restriction $\pi|_{H_1}$ is semi-simple of finite length \cite[Lemma 2.1]{tadic}. 
\item Denote by $\cO_{H_1}(\pi)$ the finite subset of $\Irr(H_1)$ consisting of distinct irreducible factors of $\pi|_{H_1}$, then $\cO_{H_1}(\pi)$ forms a single $H$-orbit in $\Irr(H_1)$ and 
\[
\pi|_{H_1}\simeq m_{H_1}(\pi) \cdot \left(\bigoplus_{\tau \in \cO_{H_1}(\pi)}\tau\right)
\]
for some $m_{H_1}(\pi) \in \ZZ_{\geq 1}$.
\item For any $\tau\in \Irr(H_1)$, there exists $\pi \in \Irr(H)$ with $\tau\subset \pi|_{H_1}$. If the central character $\omega_\tau$ of $\tau$ is unitary, then there exists such $\pi$ with unitary central character $\omega_\pi$ \cite[Prop. 2.2]{tadic}.
\item For $\pi_1,\pi_2\in \Irr(H)$, put 
\[m_{H_1}(\pi_1,\pi_2):=\left|\{\eta \in X(H/H_1)\mid \pi_1\otimes \eta \simeq \pi_2\in \Irr(H)\}\right|,
\]
then we have \cite[Prop. 2.4]{tadic}
\begin{equation}\label{eq:tadic-dimension-formula}
m_{H_1}(\pi_1,\pi_2) = \dim \Hom_{H_1}(\pi_1,\pi_2),
\end{equation}
hence the following are equivalent \cite[Corollary 2.5]{tadic}: 
\begin{enumerate}
    \item $\pi_1\otimes \eta\simeq \pi_2$ for some $\eta \in X(H/H_1)$;
    \item $\cO_{H_1}(\pi_1)\cap \cO_{H_1}(\pi_2)\neq \varnothing$;
    \item $\cO_{H_1}(\pi_1)= \cO_{H_1}(\pi_2)$.
\end{enumerate}
\end{enumerate}
In summary, the map
\[
\Irr(H)\longrightarrow \Irr(H_1)/H,\quad \pi \longmapsto \cO_{H_1}(\pi).
\]
induces a bijection 
\[\Irr(H)/X(H/H_1)\overset{\sim}{\longrightarrow} \Irr(H_1)/H.\]

We highlight a special case of \eqref{eq:tadic-dimension-formula}: if we take $\pi_1 = \pi_2 = \pi$ and if $\pi|_{H_1}$ is multiplicity free, then we get
\begin{equation}\label{eq:multiplicity-free-length}
    \text{length of }\pi|_{H_1}=\left|\{\eta \in X(H/H_1)\mid \pi\otimes \eta\simeq \pi \}\right|.
\end{equation}
 In general for $\pi \in \Irr(H)$ the restriction $\pi|_{H_1}$ may not be multiplicity free, but for $(\bH,\bH_1) = (\GL_n,\SL_n)$, a fundamental result of Tadi\'c ensures that this restriction is always multiplicity free \cite[\S 1]{tadic}. 

\subsection{Representations of $\bM$ and LLC}\label{subsec:llc}

Next we apply the above results to $(\bH,\bH_1) = (\bM^{\#},\bM)$. Since $\bM^{\#} \simeq \GL_2^3$ and $\bM \simeq (\GL_2^3)^{\det}\supset \SL_2^3$, for $\bm{\pi}^{\#} \in \Irr(M^{\#})$ the restriction $\bm{\pi}^{\#}|_{M}$ is always multiplicity free by Tadi\'c's result. Recall our notation
\[
X(M^{\#}/M):=\{\bm{\eta}^{\#} \in X(M^{\#}) \mid \bm{\eta}^{\#}|_{M}=\mathbf{1}\}\subset X(M^{\#}).
\]
For $\bm{\pi}^{\#} \in \Irr(M^{\#})$, we introduce the notation
\begin{equation}\label{eq:self-twist-group}
X_M(\bm{\pi}^{\#}):=\{\bm{\eta}^{\#} \in X(M^{\#}/M)\mid \bm{\pi}^{\#} \otimes \bm{\eta}^{\#} \simeq \bm{\pi}^{\#}\}.
\end{equation}
 Note that if we denote by $\phi_{\bm{\pi}^{\#}} \in \Phi(\bM^{\#})$ the L-parameter of $\bm{\pi}^{\#}$, then $X_M(\bm{\pi}^{\#})\subset X(M^{\#}/M)$ corresponds to the subgroup 
 \begin{equation}
    \{\bm{\eta}^{\#} \in \Hom(L_F,Z')\mid \phi_{\bm{\pi}^{\#}}\otimes \bm{\eta}^{\#}\simeq \phi_{\bm{\pi}^{\#}}\}\subset \Hom(L_F,Z') \label{eq:parameter-self-twists}
 \end{equation}
under the isomorphism $\Hom(L_F,Z')\simeq X(M^{\#}/M)$. By \eqref{eq:multiplicity-free-length}, we get 
$$|\cO_{M}(\bm{\pi}^{\#})| = (\text{ length of }\bm{\pi}^{\#}|_{M}) = |X_M(\bm{\pi}^{\#})|.$$
Denote by $m\longmapsto \overline{m}$ the natural projection $p:\widehat{M^{\#}}\longrightarrow \widehat{M}$. We have a well-defined map 
    $$ \Phi(\bM^{\#}) \longrightarrow \Phi(\bM),\qquad \phi\longmapsto \overline{\phi}:= p\circ \phi. $$ 
    
\begin{lemma}\label{lemma:parameter-lifts}
    
    For $\bm{\pi}^{\#},\bm{\pi}'^{\#}\in \Irr(M^{\#})$, 
    \[
    \overline{\phi_{\bm{\pi}^{\#}}} \simeq \overline{\phi_{\bm{\pi}'^{\#}}} \Longleftrightarrow \bm{\pi}'^{\#} \simeq \bm{\pi}^{\#}\otimes \bm{\eta}^{\#}\text{ for some }\bm{\eta}^{\#}\in X(M^{\#}/M).
    \]
\end{lemma}
\begin{proof}
    For $\phi,\phi'\in \Phi(\bM^{\#})$, if $\overline{\phi}\simeq \overline{\phi'}$, up to $\widehat{M}$-conjugacy we may assume $\overline{\phi}= \overline{\phi'}$, then $p\circ \phi = p\circ \phi'$ simply means $\phi'(w)\phi(w)^{-1}\in Z'$ for each $w\in L_F$. Define $\lambda:L_F\rightarrow Z'$ by $\lambda(w):=\phi'(w)\phi(w)^{-1}$, then this is a (continuous) homomorphism since 
    \begin{align*}
        \lambda(w_1 w_2) & = \phi'(w_1w_2) \phi(w_1 w_2)^{-1}\\
        & = \phi'(w_1)(\phi'(w_2)\phi(w_2)^{-1})\phi(w_1)^{-1} = \lambda(w_1)\lambda(w_2),
    \end{align*}
   so we have \[
    \overline{\phi_{\bm{\pi}^{\#}}} \simeq  \overline{\phi_{\bm{\pi}'^{\#}}}  \Longleftrightarrow \phi_{\bm{\pi}'^{\#}} \simeq \phi_{\bm{\pi}^{\#}}\otimes \lambda\text{ for some }\lambda\in \Hom(L_F,Z')
    \]
   and hence the result by the explanation before the lemma.
\end{proof}

For $\bm{\pi}^{\#} \in\Irr(M^{\#})$ with its L-parameter $\phi_{\bm{\pi}^{\#}}$ and $\overline{m}\in Z_{\widehat{M}}(\overline{\phi_{\bm{\pi}^{\#}}})$, pick a lift $m\in \widehat{M^{\#}}$ of $\overline{m}$ and set $$^m \phi_{\bm{\pi}^{\#}}:L_F\longrightarrow \widehat{M^{\#}},\qquad ^m\phi_{\bm{\pi}^{\#}}(w)= m\phi_{\bm{\pi}^{\#}}(w)m^{-1}.$$ By the above proof, $\overline{{}^m \phi_{\bm{\pi}^{\#}}}= \overline{\phi_{\bm{\pi}^{\#}}}$ gives $\lambda_{m}\in \Hom(L_F,Z')$ defined by $\lambda_m(w) = {}^m \phi_{\bm{\pi}^{\#}}(w) \phi_{\bm{\pi}^{\#}}(w)^{-1}$, and $^m\phi_{\bm{\pi}^{\#}} \simeq \phi_{\bm{\pi}^{\#}}\otimes \lambda_m$. Since $m$ is unique up to an element in $Z'$, the element $\lambda_m$ depends only on $\overline{m}$, hence a map 
\begin{equation}
    \label{eq:centralizer-map}Z_{\widehat{M}}(\overline{\phi_{\bm{\pi}^{\#}}}) \longrightarrow \Hom(L_F,Z'),\qquad \overline{m}\longmapsto \lambda_m.
\end{equation}
 The image of this map is clearly the subgroup \eqref{eq:parameter-self-twists} (since $^m \phi_{\bm{\pi}^{\#}}\simeq \phi_{\bm{\pi}^{\#}}$), which is isomorphic to $X_M(\bm{\pi}^{\#})$ under $\Hom(L_F,Z')\simeq X(M^{\#}/M)$.

\begin{lemma}\label{lemma:component-group-self-twists}
    For $\bm{\pi}^{\#} \in\Irr(M^{\#})$ with its L-parameter $\phi_{\bm{\pi}^{\#}}$, the above map \eqref{eq:centralizer-map} induces an isomorphism 
    \begin{equation}
        \label{eq:component-group-for-llc}\cS(\overline{\phi_{\bm{\pi}^{\#}}}):=\pi_0\left(Z_{\widehat{M}}(\overline{\phi_{\bm{\pi}^{\#}}})/Z(\widehat{M})\right)\xrightarrow{\sim} X_M(\bm{\pi}^{\#}).
    \end{equation}

\end{lemma}

\begin{proof}
By definition, the kernel of the map \eqref{eq:centralizer-map} is the image of 
\(Z_{\widehat{ M^{\#}}}(\phi_{\bm{\pi}^{\#}})\) in \(Z_{\widehat{ M}}(\overline{\phi_{\bm{\pi}^{\#}}})\), hence an exact sequence
\[
1\longrightarrow
Z_{\widehat{ M^{\#}}}(\phi_{\bm{\pi}^{\#}})/Z'
\longrightarrow
Z_{\widehat{M}}(\overline{\phi_{\bm{\pi}^{\#}}})
\longrightarrow
X_M(\bm{\pi}^{\#})
\longrightarrow 1.
\]
By $p^{-1}\bigl(Z(\widehat{M})\bigr)=Z(\widehat{ M^{\#}})$, taking quotient along the above exact sequence gives 
\[
1\longrightarrow
Z_{\widehat{ M^{\#}}}(\phi_{\bm{\pi}^{\#}})/Z(\widehat{ M^{\#}})
\longrightarrow
Z_{\widehat{M}}(\overline{\phi_{\bm{\pi}^{\#}}})/Z(\widehat{M})
\longrightarrow
X_M(\bm{\pi}^{\#})
\longrightarrow 1.
\]
Since $Z_{\widehat{ M^{\#}}}(\phi_{\bm{\pi}^{\#}})$ and $Z_{\widehat{ M^{\#}}}(\phi_{\bm{\pi}^{\#}})/Z(\widehat{ M^{\#}})$ are connected (by $\bM^{\#}\simeq \GL_2^3$), we get
\[
\pi_0\!\left(
Z_{\widehat{M}}(\overline{\phi_{\bm{\pi}^{\#}}})/Z(\widehat{M})
\right)
\simeq X_M(\bm{\pi}^{\#}).
\]
\end{proof}

Now for $\bm{\tau}\in \Irr(M)$, we may pick any $\bm{\pi}^{\#} \in \Irr(M^{\#})$ with $\bm{\pi}^{\#}|_M\supset \bm{\tau}$, and define the L-parameter of $\bm{\tau}$ to be $\phi_{\bm{\tau}}:= \overline{\phi_{\bm{\pi}^{\#}}} \in \Phi(\bM)$. If $\bm{\pi}'^{\#}\in\Irr(M^{\#})$ is another representation whose restriction contains $\bm{\tau}$, we have $\bm{\pi}'^{\#}\simeq\bm{\pi}^{\#}\otimes\bm{\eta}^{\#}$ for some $\bm{\eta}^{\#}\in X(M^{\#}/M)$ and the preceding lemma shows $\overline{\phi_{\bm{\pi}'^{\#}}}\simeq\overline{\phi_{\bm{\pi}^{\#}}}$. Thus the L-packet of $\bm{\tau}$ is
\[
\Pi_{\phi_{\bm{\tau}}}:=\{\bm{\tau}'\in\Irr(M)\mid \bm{\tau}'\subset \bm{\pi}^{\#}|_M\}=\cO_M(\bm{\pi}^{\#})
\]
for any $\bm{\pi}^{\#}\in \Irr(M^{\#})$ with $\bm{\pi}^{\#}|_M\supset \bm{\tau}$.

\subsection{Restriction of irreducible principal series}\label{subsec:principal-series}

For $\tau\in \Irr(G)$, we denote by $X(\tau)$ the group of characters $\eta:F^{\times}\rightarrow \CC^\times$ such that $\tau\otimes \eta\simeq \tau$. Let $\mu,\nu: F^\times \rightarrow \CC^\times$ be two characters with $\mu\nu^{-1}\neq |\cdot|^{\pm 1}$, so that $\tau:=\mu\times \nu$ is an irreducible principal series. Then for a character $\eta:F^\times \rightarrow \CC^\times$, we have $\tau\otimes \eta \simeq \eta\mu\times \eta\nu$, so that
\[
X(\tau)=\begin{cases}
    \{\mathbf{1}\},& \mu\nu^{-1}\text{ is not of order 2},\\
    \{\mathbf{1}, \mu\nu^{-1}\},& \mu\nu^{-1}\text{ is of order 2}.
\end{cases}
\]

We next take characters $\mu_i,\nu_i:F^\times \longrightarrow \CC^\times$ for $1\leq i \leq 3$ and use the notation $\xi_i:=\mu_i \nu_i^{-1}$. Assume that $\xi_i\neq |\cdot|^{\pm 1}$ and form the irreducible representations $$\pi_i:=\mu_i\times \nu_i \in \Irr(G)\qquad \text{and}\qquad\bm{\pi}^{\#}:=\pi_1\boxtimes \pi_2\boxtimes \pi_3\in \Irr(M^{\#}).$$ 
Recall that in \S~\ref{subsec:notation} we introduced the notation
\[
\bm{\chi}^{\#}:=(\mu_1,\nu_1;\mu_2,\nu_2;\mu_3,\nu_3)
\]
for the corresponding character of $T_M^{\#}$, and $\bm{\chi}:=\bm{\chi}^{\#}|_{T_M}$ its restriction to $T_M$.  From the relations
\[
B_M = B_M^{\#}\cap M,\qquad M^{\#} = B_M^{\#}\cdot M,\qquad \delta_{B_M^{\#}}|_{B_M} = \delta_{B_M},
\]
we find that the restriction map induces an $M$-isomorphism 
\[
\bm{\pi}^{\#}|_{M} = I({\bm{\chi}}^{\#})|_{M} \overset{\sim}{\longrightarrow} I({\bm{\chi}}).
\]

\begin{lemma}\label{lemma:character-restriction}
Let
\[
\bm{\chi}^{\#}
=
(\mu_1,\nu_1;\mu_2,\nu_2;\mu_3,\nu_3),
\qquad
\bm{\chi}'^{\#}
=
(\mu'_1,\nu'_1;\mu'_2,\nu'_2;\mu'_3,\nu'_3)
\]
be two characters of \(T_{M}^{\#}\). Then $\left.\bm{\chi}^{\#}\right|_{T_M}
=
\left.\bm{\chi}'^{\#}\right|_{T_M}$ if and only if 
\begin{itemize}
    \item there exist $\eta_1,\eta_2,\eta_3:F^\times \longrightarrow \CC^\times$ with $\eta_1\eta_2\eta_3=\mathbf{1}$, such that 
    \[
    \mu_i' = \mu_i \eta_i,\qquad \nu_i' = \nu_i \eta_i\qquad \text{ for }1\leq i\leq 3.
    \]
\end{itemize}
\end{lemma}

\begin{proof}
Direct computation. It suffices to check  $\left.\bm{\chi}^{\#}\right|_{T_M}=\mathbf{1}$ if and only if $\mu_i = \nu_i$ and $\mu_1\mu_2\mu_3=\mathbf{1}$ (then apply this to $\bm{\chi}'^{\#}(\bm{\chi}^{\#})^{-1}$). Using the coordinates  \[
    t=t(a_1,a_2,a_3,c):=\left(\begin{bmatrix}
        a_1 & \\ & ca_1^{-1}
    \end{bmatrix},\begin{bmatrix}
        a_2 & \\ & ca_2^{-1}
    \end{bmatrix},\begin{bmatrix}
        a_3 & \\ & ca_3^{-1}
    \end{bmatrix}\right),\qquad a_i,c\in F^\times,
    \]
we have
\[
{\bm{\chi}}(t(a_1,a_2,a_3,c))
=
\prod_{i=1}^3 \mu_i(a_i)\nu_i(ca_i^{-1})
=
(\nu_1\nu_2\nu_3)(c)\prod_{i=1}^3 (\mu_i\nu_i^{-1})(a_i),
\]
and hence the result.
\end{proof}

For $1\leq i\leq 3$ we denote by $X_i:=X(\pi_i)$, then the definition \eqref{eq:self-twist-group} can be explicated as
\begin{equation}\label{eq:principal-series-self-twists}
X_M(\bm{\pi}^{\#})=\{(\eta_1,\eta_2,\eta_3)\in X_1\times X_2\times X_3\mid \eta_1 \eta_2 \eta_3 =\mathbf{1}\}.
\end{equation}
From $|X_i|\leq 2$ we find that the order of $X_M(\bm{\pi}^{\#})$ can only be $1,2,4$. We can give a very explicit classification when $X_M(\bm{\pi}^{\#})$ is of order 2 or 4.

\begin{lemma}\label{lemma:order-four-classification}
    Suppose $|X_M(\bm{\pi}^{\#})|=4$, then we have $\xi_1= \xi_2=\xi_3=\xi$ for some character $\xi:F^\times \longrightarrow \CC^1$ of order 2, and 
    \[
    X_M(\bm{\pi}^{\#})=\{\mathbf{1},(\xi,\xi,\mathbf{1}),(\xi,\mathbf{1},\xi),(\mathbf{1},\xi,\xi)\}\subset X_1\times X_2\times X_3.
    \]
    If moreover $\omega_{\pi_1}\omega_{\pi_2}\omega_{\pi_3}=\mathbf{1}$, then there exists a character $\rho: F^\times \longrightarrow \CC^1$ with $\rho^2=\xi$.
\end{lemma}
\begin{proof}
    It is obvious that $|X_M(\bm{\pi}^{\#})|=4$ implies the order of $X_1\times X_2\times X_3$ must be 8, so that $\xi_1,\xi_2,\xi_3$ are all of order 2. If $\xi_1\neq \xi_2$, then the following four elements
    \[
    (\xi_1,\xi_2,\mathbf{1}),\quad (\xi_1,\mathbf{1},\mathbf{1}),\quad(\mathbf{1},\xi_2,\mathbf{1}),\quad (\mathbf{1},\mathbf{1},\xi_3)
    \]
    do not belong to $X_M(\bm{\pi}^{\#})$. But the remaining four elements 
    \[
    \mathbf{1},\quad(\xi_1,\mathbf{1},\xi_3),\quad(\mathbf{1},\xi_2,\xi_3),\quad (\xi_1,\xi_2,\xi_3)
    \]
    do not form a group, hence we must have $\xi_1=\xi_2$. By cyclic symmetry, we also have $\xi_2=\xi_3$. 

    We have $\omega_{\pi_i}=\mu_i \nu_i$, so from $\omega_{\pi_1}\omega_{\pi_2}\omega_{\pi_3}=\mathbf{1}$ and $\xi^2=\mathbf{1}$ we would get 
    \[
    \mathbf{1} =\prod_{i=1}^3\mu_i \nu_i = \prod_{i=1}^3 \nu_i^2 \xi_i = \xi\cdot \prod_{i=1}^3 \nu_i^2,
    \]
    so that $\xi = \rho^2$ for $\rho = (\nu_1\nu_2\nu_3)^{-1}$.
\end{proof}

\begin{corollary} \label{corollary:order-4-canonical-form}
    If $|X_M(\bm{\pi}^{\#})|=4$ and $\omega_{\pi_1}\omega_{\pi_2}\omega_{\pi_3}=\mathbf{1}$, then $\bm{\pi}^{\#}|_M\simeq \bm{\pi}'^{\#}|_M$ for some 
    \[
    \bm{\pi}'^{\#} = (\rho\times \rho^{-1})\boxtimes(\rho\times \rho^{-1})\boxtimes(\rho\times \rho^{-1}),
    \]
    with $\rho:F^\times \longrightarrow \CC^1$ of order 4.
\end{corollary}
\begin{proof}
    By the computation in the proof of the previous lemma, we have $\bm{\pi}^{\#} = I(\bm{\chi}^{\#})$ with $\bm{\chi}^{\#}$ of the form 
    \[
    \bm{\chi}^{\#} = (\nu_1\xi,\nu_1;\nu_2\xi,\nu_2;\nu_3\xi,\nu_3),
    \]
    where $\xi:F^\times \longrightarrow \CC^\times$ is of order 2 and $\xi(\nu_1\nu_2\nu_3)^2=\mathbf{1}$. Take $\rho=\nu_1\nu_2\nu_3$ and apply Lemma \ref{lemma:character-restriction} to $\bm{\chi}^{\#}$ and $\bm{\chi}'^{\#} = (\rho,\rho^{-1};\rho,\rho^{-1};\rho,\rho^{-1})$ with $\eta_i = (\nu_i\rho)^{-1}$, $1\leq i\leq 3$.
\end{proof}

\begin{lemma}\label{lemma:order-two-classification}
    Suppose $|X_M(\bm{\pi}^{\#})|=2$, then up to permutation exactly one of the following holds.
    \begin{enumerate}
        \item\label{item:order-two-type-i} We have $\xi_1 = \xi_2 = \xi$ of order 2, and $\xi_3 \neq \xi$. In this case, we get $$X_M(\bm{\pi}^{\#})=\{\mathbf{1},(\xi,\xi,\mathbf{1})\}.$$
        \item\label{item:order-two-type-ii} The three characters $\xi_1,\xi_2,\xi_3$ are all of order 2, distinct from each other, and $\xi_1\xi_2\xi_3 = \mathbf{1}$. In this case, we get $$X_M(\bm{\pi}^{\#}) =\{\mathbf{1}, (\xi_1,\xi_2,\xi_3)\}.$$
    \end{enumerate}
\end{lemma}

\begin{proof}
    We set 
    \[
    S:=\{1\leq i\leq 3\mid \xi_i \text{ is of order }2\},
    \]
    then clearly $|S|=1$ implies $|X_M(\bm{\pi}^{\#})|=1$. If $|S|=2$, up to permutation say $\xi_1,\xi_2$ are of order 2, then the only possibility is $X_M(\bm{\pi}^{\#})=\{\mathbf{1},(\xi_1,\xi_2,\mathbf{1})\}$, which forces $\xi_1=\xi_2$. If $|S|=3$, we have two cases:
    \begin{itemize}
        \item if $\xi_1,\xi_2,\xi_3$ are distinct from each other, then the only possibility is $X_M(\bm{\pi}^{\#})=\{\mathbf{1},(\xi_1,\xi_2,\xi_3)\}$, which forces $\xi_1\xi_2\xi_3 =\mathbf{1}$;
        \item otherwise up to permutation $\xi_1 = \xi_2 = \xi$, then $(\xi,\xi,\mathbf{1})\in X_M(\bm{\pi}^{\#})$, which forces $\xi_3\neq \xi$.
    \end{itemize}
\end{proof}

We introduce the following terminology:
\begin{itemize}
    \item if up to permutation part (\ref{item:order-two-type-i}) (resp. part (\ref{item:order-two-type-ii})) of the above lemma holds, we say $\{\bm{\sigma}_+, \bm{\sigma}_-\}$ is of \textbf{type I} (resp. \textbf{type II}).
\end{itemize}

\begin{corollary}\label{corollary:order-2-canonical-form}
Assume $|X_M(\bm{\pi}^{\#})|=2$ and $\omega_{\pi_1}\omega_{\pi_2}\omega_{\pi_3}=\mathbf{1}$.
\begin{enumerate}
    \item For type I case, up to permutation suppose $\xi_1=\xi_2 = \xi$ is of order 2 and $\xi_3\neq \xi$, then there exists a character $\alpha:F^\times \longrightarrow \CC^{\times}$ with $\alpha^2=\xi_3$, such that 
    \[
    \bm{\pi}'^{\#}:=(\xi\times \mathbf{1})\boxtimes (\xi\times \mathbf{1})\boxtimes (\alpha\times \alpha^{-1})
    \]
    satisfies $\bm{\pi}'^{\#}|_M\simeq \bm{\pi}^{\#}|_M$.

    \item For type II case, the three characters $\xi_1,\xi_2,\xi_3$ are all of order 2 and mutually distinct, with $\xi_1\xi_2\xi_3=\mathbf{1}$. There exists a character $\lambda:F^\times \longrightarrow \CC^1$ of order $\leq 2$, such that 
    \[
    \bm{\pi}'^{\#} := (\xi_1\times \mathbf{1})\boxtimes (\xi_2\times \mathbf{1})\boxtimes (\lambda\xi_3\times \lambda)
    \]
    satisfies $\bm{\pi}'^{\#}|_M\simeq \bm{\pi}^{\#}|_M$.
\end{enumerate}
\end{corollary}

\begin{proof}
    By the previous Lemma, in the type I case we have 
    \[\bm{\chi}^{\#}=(\xi\nu_1,\nu_1;\xi\nu_2,\nu_2;\xi_3\nu_3,\nu_3)\]
    for $\xi:F^\times \longrightarrow\CC^\times$ of order 2 and $\xi_3\neq \xi$. The central character condition reads
    \[\mathbf{1}=\omega_{\pi_1}\omega_{\pi_2}\omega_{\pi_3} = (\xi\nu_1\nu_2\nu_3)^2\xi_3 = (\nu_1\nu_2\nu_3)^2\xi_3.\]
    We set $\alpha:=\nu_1\nu_2\nu_3 \xi_3$ (with $\alpha^2=\xi_3$), and apply Lemma \ref{lemma:character-restriction} to $\bm{\chi}^{\#}$ and $\bm{\chi}'^{\#}=(\xi,\mathbf{1};\xi,\mathbf{1};\alpha,\alpha^{-1})$ with $\eta_1 = \nu_1^{-1}$, $\eta_2 = \nu_2^{-1}$, $\eta_3 = \nu_1\nu_2$. 

    The type II case is a similar computation. We have 
    \[\bm{\chi}^{\#}=(\xi_1\nu_1,\nu_1;\xi_2\nu_2,\nu_2;\xi_3\nu_3,\nu_3)\]
    with $\xi_1\xi_2\xi_3=\mathbf{1}$ and 
    \[\mathbf{1}=\omega_{\pi_1}\omega_{\pi_2}\omega_{\pi_3} = (\nu_1\nu_2\nu_3)^2(\xi_1\xi_2\xi_3) = (\nu_1\nu_2\nu_3)^2.\]
    We set $\lambda:=\nu_1\nu_2\nu_3$, which is either trivial or of order 2, and apply Lemma \ref{lemma:character-restriction} to $\bm{\chi}^{\#}$ and $\bm{\chi}'^{\#}=(\xi_1,\mathbf{1};\xi_2,\mathbf{1};\lambda \xi_3,\lambda)$ with $\eta_1 = \nu_1^{-1}$, $\eta_2 = \nu_2^{-1}$, $\eta_3 = \nu_1\nu_2$. 
\end{proof}

\subsection{The R-group}\label{subsec:r-group}
We have an identification 
\[
W(\bM^{\#},\bT_M^{\#})\simeq W(\bM,\bT_M)=: W \simeq (\ZZ/2\ZZ)^3
\]
of Weyl groups. For $1\leq i \leq 3$, we denote by $w_i \in W$ the simple reflection at the $i$-th $\GL_2$-factor, and for a subset $S$ of $\{1,2,3\}$ we write $w_S:=\prod_{i\in S}w_i$ (and $w_{\varnothing}:=1$ by definition). To simplify the notations, if we have a concrete $S$ such as $S=\{1,2\}$, we shall simply write $w_S$ as $w_{12}$. Thus the longest element in $W$ would be denoted by $w_{123}$.

Now following the setting in the previous subsection, let ${\bm{\chi}}^{\#}:T_M^{\#}\longrightarrow \CC^\times$ be a character of $T_M^{\#}$ such that $I(\bm{\chi}^{\#})$ is irreducible, and $\bm{\chi}:=\bm{\chi}^{\#}|_{T_M}$. We set $$S_{\bm{\chi}}:=\{1\leq i\leq 3\mid w_i {\bm{\chi}} = {\bm{\chi}}\}$$ and define
\[
W_{{\bm{\chi}}}:=\Stab_{W}({\bm{\chi}}),\qquad W_{{\bm{\chi}}}^0:=\langle w_i\mid i \in S_{{\bm{\chi}}}\rangle,\qquad R_{\bm{\chi}}:=W_{\bm{\chi}}/W_{\bm{\chi}}^0.
\]
The quotient $R_{\bm{\chi}}$ is a generalization of the Knapp-Stein R-group, see for example, \cite{luo}. For explicit computation, we still choose a lift 
\[\bm{\chi}^{\#}:=(\mu_1,\nu_1;\mu_2,\nu_2;\mu_3,\nu_3)\]
of $\bm{\chi}$ (i.e. $\bm{\chi}^{\#}|_{T_M} = \bm{\chi}$) and keep the notation $\xi_i:=\mu_i\nu_i^{-1}$.

\begin{proposition}\label{proposition:r-group-self-twists}
Setting and notations as above, we then have
   \[
    W_{\bm{\chi}}=\{w_S\mid  \xi_i^2=\mathbf{1}\text{ for } i\in S,\text{ and }\prod_{i\in S} \xi_i =\mathbf{1}\}.
   \]
   The homomorphism $\Phi: W_{{\bm{\chi}}}\longrightarrow X_M(\bm{\pi}^{\#})$ given by 
   \[
   \Phi(w_S):=(\eta_{1},\eta_{2},\eta_{3}),\qquad \eta_{i}:=\begin{cases}
       \xi_i,& i\in S,\\ \mathbf{1},& i\notin S
   \end{cases}
   \]
   induces a canonical isomorphism $R_{\bm{\chi}} = W_{{\bm{\chi}}}/W_{{\bm{\chi}}}^0\overset{ \sim}{\longrightarrow}X_M(\bm{\pi}^{\#})$.
\end{proposition}

\begin{proof}
 
    Direct computation. Using the coordinates 
    \[
    t=t(a_1,a_2,a_3,c):=\left(\begin{bmatrix}
        a_1 & \\ & ca_1^{-1}
    \end{bmatrix},\begin{bmatrix}
        a_2 & \\ & ca_2^{-1}
    \end{bmatrix},\begin{bmatrix}
        a_3 & \\ & ca_3^{-1}
    \end{bmatrix}\right)
    \]
   and recall our notation $\xi_i:=\mu_i \nu_i^{-1}$, we have
\[
{\bm{\chi}}(t)
=
\prod_{i=1}^3 \mu_i(a_i)\nu_i(ca_i^{-1})
=
(\nu_1\nu_2\nu_3)(c)\prod_{i=1}^3 \xi_i(a_i)
\]
and
\[
w_S{\bm{\chi}}(t)
=
\left(\prod_{i\in S}\mu_i(c)\xi_i^{-1}(a_i)\right)
\left(\prod_{i\notin S}\nu_i(c)\xi_i(a_i)\right),
\]
so that
\[
\frac{w_S{\bm{\chi}}(t)}{{\bm{\chi}}(t)}
=\prod_{i\in S}
\frac{\mu_i(c)\xi_i^{-1}(a_i)}{\nu_i(c)\xi_i(a_i)}
=\prod_{i\in S}\xi_i(c)\xi_i^{-2}(a_i).
\]
Thus 
\begin{align*}
    w_S {\bm{\chi}} = {\bm{\chi}}&\Longleftrightarrow \prod_{i\in S} \xi_i(c)\xi_i^{-2}(a_i)=1\text{ for all }c,a_i\in F^\times\\
    &\Longleftrightarrow \xi_i^2=\mathbf{1}\text{  for } i\in S,\text{ and }\prod_{i\in S} \xi_i =\mathbf{1}.
\end{align*}
Clearly $\Phi$ is surjective, and
\[
\ker(\Phi) =\{w_S\mid \xi_i =\mathbf{1} \text{ for }i\in S\}=\langle w_i\mid i\in S_{{\bm{\chi}}}\rangle = W_{{\bm{\chi}}}^0,
\]
hence the isomorphism $R_{\bm{\chi}} = W_{{\bm{\chi}}}/W_{{\bm{\chi}}}^0\overset{ \sim}{\longrightarrow}X_M(\bm{\pi}^{\#})$.
\end{proof}

We thus identify the ``generalized'' R-group $R_{{\bm{\chi}}}$ with $X_M(\bm{\pi}^{\#})$ via this isomorphism.  Next we define 
\begin{equation}\label{eq:standard-whittaker-character}\psi_{U_M}: U_M\longrightarrow \CC^1,\qquad (u(x_1),u(x_2),u(x_3))\longrightarrow \psi(x_1+x_2+x_3),\end{equation}
and fix the Whittaker datum 
\begin{equation}\label{eq:standard-whittaker-datum}\mathfrak{w}_M:=(B_M,\psi_{U_M})\end{equation}
 of the group $M$. 
 
 \begin{remark}
    We should note that $\mathfrak{w}_M$ is by definition the $M$-conjugacy class of the pair $(B_M,\psi_{U_M})$, thus the Whittaker datum $\mathfrak{w}_M$ does not depend on the choice of the character $\psi: F\longrightarrow \CC^1$ (since we can conjugate it by an element of the form $(\diag(a,1),\diag(a,1),\diag(a,1))$). We can also easily check that the number of different Whittaker data of $M$ is 
    \[ [(F^\times)^3:(F^\times)^{\Delta} (F^{\times 2})^3] = [F^\times: F^{\times 2}]^2.\]
    Let us also note that $\mathfrak{w}_M$ admits another representative $(\overline{B_M},\psi_{\overline{U_M}})$, where $\overline{B_M}$ is the opposite lower triangular Borel and 
    \[ \psi_{\overline{U_M}}(\overline{u}(x_1),\overline{u}(x_2),\overline{u}(x_3)):=\psi(x_1+x_2+x_3).\]
    For this it suffices to conjugate $(B_M,\psi_{U_M})$ by $s_0^{\Delta} = (s_0,s_0,s_0)$.
 \end{remark}

 We fix the following choice of representatives for Weyl group elements:
\begin{itemize}
    \item for each $w_i$ we set $$\dot{w}_1:=\left(\begin{bmatrix}
        & 1 \\ -1 &
    \end{bmatrix},I_2,I_2\right)$$
and similarly for $\dot{w}_2,\dot{w}_3$, then we set $\dot{w}_S:=\prod_{i\in S} \dot{w}_i$ for a subset $S\subset \{1,2,3\}$.
\end{itemize}
For $w\in W$, we set
\[ U_w:= U_M\cap w^{-1}\overline{U_M}w,\]
and the standard intertwining operator is defined by the integral
\[
J(w,{\bm{\chi}}):I({\bm{\chi}})\longrightarrow I(w{\bm{\chi}}),\qquad
(J(w,{\bm{\chi}})f)(g)=\int_{U_w}f(\dot{w}^{-1}ug)\,du.
\]
The integral converges only in a suitable positive chamber, and admits meromorphic continuation to general $\bm{\chi}$. Similarly, the Jacquet integral
\[
\ell_{\bm{\chi}}^{\psi}:I(\bm{\chi})\longrightarrow \CC,\qquad\ell_{\bm{\chi}}^{\psi}(f)=\int_{U_M}f(\dot{w}_{123}^{-1}u)\psi_{U_M}(u)^{-1}\,du
\]
is also first defined in a suitable positive chamber and, in the compact realization of each inertial family, admits holomorphic continuation in $\bm{\chi}$; see \cite[\S 4.1]{delorme} (especially (4.1) and what follows). The local coefficient \(C^{\psi}(w,{\bm{\chi}})\in \CC\) as a meromorphic function in $\bm{\chi}$ is defined by
\[
\ell_{w{\bm{\chi}}}^{\psi}\circ J(w,{\bm{\chi}})=C^{\psi}(w,{\bm{\chi}})\ell_{\bm{\chi}}^{\psi}
\]
according to multiplicity one of the Whittaker model. The normalized intertwining operator is then defined to be
\[
R(w,{\bm{\chi}}):=C^{\psi}(w,{\bm{\chi}})^{-1}J(w,{\bm{\chi}}).
\]
characterized by
\begin{equation}\label{eq:whittaker-normalization}
\ell_{w{\bm{\chi}}}^{\psi}\circ R(w,{\bm{\chi}})=\ell_{\bm{\chi}}^{\psi}.
\end{equation}

Assume henceforth that $\bm\chi$ is unitary. By the theory of R-groups (see for example \cite{arthur}), $R(w,\bm\chi)$ is then regular and gives a unitary operator for $w\in W$, and we have the cocycle relation
\begin{equation} \label{eq:r-group-cocycle}
R(w',w{\bm{\chi}})\circ R(w,{\bm{\chi}})
=
R(w'w,{\bm{\chi}}),\qquad w,w'\in W
\end{equation}
by our Whittaker normalization. Indeed, on the generic locus the relevant intertwining space is one-dimensional, and Whittaker uniqueness together with the normalization fixes the scalar; the identity then holds everywhere by meromorphic continuation. For the general R-group formalism, see \cite[pp.~86--87]{arthur}.
For clarity, we state the following as a Lemma. 

\begin{lemma}\label{lemma:weyl-kernel-intertwining} We have
\[
R(w,{\bm{\chi}})=\operatorname{id}_{I({\bm{\chi}})},\qquad w\in W_{\bm{\chi}}^0.
\]
\end{lemma}
\begin{proof}
It suffices to prove this for $w_i$ with $i\in S_{\bm{\chi}}$. The condition $w_i{\bm{\chi}}={\bm{\chi}}$ implies $\xi_i=\mathbf{1}$, and hence $\mu_i=\nu_i$. Under the restriction isomorphism
\[
I({\bm{\chi}}^{\#})|_M\simeq I({\bm{\chi}}),
\]
the operator $R(w_i,{\bm{\chi}})$ is the restriction of the tensor product of the identity operators on the other two factors with the normalized self-intertwining operator on $\mu_i\times \mu_i$. Since $\mu_i\times \mu_i$ is irreducible, this rank-one self-intertwining operator is a scalar by Schur's lemma, then it must be 1 by the normalization \eqref{eq:whittaker-normalization}. Thus $R(w_i,{\bm{\chi}})=\operatorname{id}_{I({\bm{\chi}})}$, and the claim for general $w\in W_{\bm{\chi}}^0$ follows from the cocycle relation \eqref{eq:r-group-cocycle}. 
\end{proof}

Consequently, we have
\[
R(wu,{\bm{\chi}})
=
R(w,u{\bm{\chi}})\circ R(u,{\bm{\chi}})
=
R(w,{\bm{\chi}}),\qquad w\in W_{\bm{\chi}},~u\in W_{\bm{\chi}}^0,
\]
and for $r\in R_{\bm{\chi}}=W_{\bm{\chi}}/W_{\bm{\chi}}^0$ we shall write directly $R(r,{\bm{\chi}})$. The cocycle relation shows that
\[
R_{\bm{\chi}}\longrightarrow \operatorname{End}_M(I({\bm{\chi}}))^\times,\qquad
r\longmapsto R(r,{\bm{\chi}})
\]
is a group homomorphism, and Harish-Chandra's commuting algebra theorem (see for example, \cite[pp. 86]{arthur}) tells us that this induces a $\mathbb{C}$-algebra isomorphism
\[
\mathbb{C}[R_{\bm{\chi}}]\overset{\sim}{\longrightarrow}\operatorname{End}_M(I({\bm{\chi}})),
\]
or in other words,
\[
\operatorname{End}_M(I({\bm{\chi}}))
=
\bigoplus_{r\in R_{\bm{\chi}}}\mathbb{C}R(r,{\bm{\chi}})
\simeq\mathbb{C}[R_{\bm{\chi}}].
\]
For $\kappa\in\Irr(R_{\bm{\chi}})$, we define the primitive idempotent operator
\[
e_\kappa:=\frac{1}{|R_{\bm{\chi}}|}\sum_{r\in R_{\bm{\chi}}}\kappa(r)R(r,{\bm{\chi}})\in\operatorname{End}_M(I({\bm{\chi}})),
\]
and
\[
\bm{\sigma}_\kappa:=e_\kappa I({\bm{\chi}})
=
\left\{v\in I({\bm{\chi}})\mid R(r,{\bm{\chi}})v=\kappa(r)^{-1}v
\text{ for every }r\in R_{\bm{\chi}}\right\}.
\]
Then the theory of R-groups gives the decomposition
\[
I({\bm{\chi}})
\simeq
\bigoplus_{\kappa\in\Irr(R_{\bm{\chi}})}
\kappa^\vee\boxtimes\bm{\sigma}_\kappa
\]
of $I({\bm{\chi}})$ as a representation of $R_{\bm{\chi}}\times M$ \cite[pp. 87]{arthur}. By the normalization \eqref{eq:whittaker-normalization} we have
\[
\ell_{\bm{\chi}}^\psi\circ R(r,{\bm{\chi}})
=
\ell_{\bm{\chi}}^\psi,\qquad r\in R_{\bm{\chi}},
\]
so that
\[
\ell_{\bm{\chi}}^\psi(v)
=
\ell_{\bm{\chi}}^\psi\bigl(R(r,{\bm{\chi}})v\bigr)
=
\kappa(r)^{-1}\ell_{\bm{\chi}}^\psi(v),\qquad v\in \bm{\sigma}_{\kappa}
\]
and thus $\ell_{\bm{\chi}}^\psi|_{\bm{\sigma}_\kappa}$ vanishes unless $\kappa=\mathbf{1}$. In particular, since $I(\bm{\chi})$ is multiplicity-free, $\bm{\sigma}_{\mathbf{1}}$ is the unique $\mathfrak{w}_M$-generic constituent.

\subsection{Compatibility with restriction}\label{subsec:restriction-compatibility}

We first record the rank-one normalization that will be used on each factor of $M^{\#}$. We equip $F$ with the self-dual Haar measure with respect to $\psi$. Given characters $\mu,\nu: F^\times\longrightarrow \CC^\times$, the standard intertwining operator
\[
J_{\mu,\nu}:I_{\mu,\nu}\longrightarrow I_{\nu,\mu}
\]
is initially defined in the usual convergence chamber by
\[
(J_{\mu,\nu}f)(g):=\int_F f(\dot{w}_0^{-1}u(x)g)\,dx
\]
and is understood elsewhere by meromorphic continuation. The Jacquet integral
\[
\ell_{\mu,\nu}^{\psi}(f):=\int_F f(\dot{w}_0^{-1}u(x))\psi(-x)\,dx
\]
is likewise initially defined in a suitable chamber and admits holomorphic continuation. Let $C^{\psi}(\mu,\nu)$ be the local coefficient determined by
\[
\ell_{\nu,\mu}^{\psi}\circ J_{\mu,\nu}
=C^{\psi}(\mu,\nu)\ell_{\mu,\nu}^{\psi}.
\]
We define the Whittaker-normalized rank-one intertwining operator by
\begin{equation}\label{eq:rank-one-normalized-intertwining}
R_{\mu,\nu}:=C^{\psi}(\mu,\nu)^{-1}J_{\mu,\nu}:I_{\mu,\nu}\longrightarrow I_{\nu,\mu}.
\end{equation}
Thus
\begin{equation}\label{eq:rank-one-whittaker-normalization}
\ell_{\nu,\mu}^{\psi}\circ R_{\mu,\nu}=\ell_{\mu,\nu}^{\psi}.
\end{equation}

For a character
\[
\bm{\chi}^{\#}=(\mu_1,\nu_1;\mu_2,\nu_2;\mu_3,\nu_3)
\]
 of $T_M^{\#}$, we have
\begin{equation}\label{eq:principal-series-tensor-realization}
I(\bm{\chi}^{\#})
\simeq I_{\mu_1,\nu_1}\boxtimes I_{\mu_2,\nu_2}\boxtimes I_{\mu_3,\nu_3}.
\end{equation}
For a subset $S$ of $\{1,2,3\}$, let
\[
R(w_S,\bm{\chi}^{\#}):I(\bm{\chi}^{\#})
\longrightarrow I(w_S\bm{\chi}^{\#})
\]
be our normalized intertwining operator for $M^{\#}= \GL_2^3(F)$. For $1\leq i\leq3$, we set
\[
R_i:=
\begin{cases}
R_{\mu_i,\nu_i}:I_{\mu_i,\nu_i}\longrightarrow I_{\nu_i,\mu_i},&i\in S,\\
\operatorname{id}_{I_{\mu_i,\nu_i}},&i\notin S.
\end{cases}
\]
Under the isomorphism \eqref{eq:principal-series-tensor-realization} we have
\[
J(w_S,\bm{\chi}^{\#})=J_1\otimes J_2\otimes J_3,\qquad J_i:=\begin{cases}
    J_{\mu_i,\nu_i}, & i\in S,\\
    \mathrm{id},& i\notin S.
\end{cases}
\]
The Whittaker functional also admits a tensor decomposition
\[
\ell_{\bm{\chi}^{\#}}^{\psi}
=\ell_{\mu_1,\nu_1}^{\psi}\otimes
\ell_{\mu_2,\nu_2}^{\psi}\otimes
\ell_{\mu_3,\nu_3}^{\psi},
\]
so that the local coefficient decomposes as
\[
C^{\psi}(w_S,\bm{\chi}^{\#})
=\prod_{i\in S}C^{\psi}(\mu_i,\nu_i).
\]
Consequently, the normalized intertwining operator also factors as
\begin{equation}\label{eq:normalized-intertwining-tensor-factorization}
R(w_S,\bm{\chi}^{\#})
=R_1\otimes R_2\otimes R_3.
\end{equation}
Recall that for $\bm{\chi}:=\bm{\chi}^{\#}|_{T_M}$, restriction of functions gives an $M$-isomorphism
\[
\operatorname{res}_{\bm{\chi}^{\#}}:I(\bm{\chi}^{\#})\xrightarrow{\sim}I(\bm{\chi}),
\qquad f\longmapsto f|_M.
\]

\begin{lemma}\label{lemma:normalized-intertwining-restriction}

For $w\in W$, the character $w\bm{\chi}^{\#}$ restricts to $w\bm{\chi}$, and we have
\begin{equation}\label{eq:normalized-intertwining-restriction}
\operatorname{res}_{w\bm{\chi}^{\#}}\circ R(w,\bm{\chi}^{\#})
=R(w,\bm{\chi})\circ\operatorname{res}_{\bm{\chi}^{\#}}
\end{equation}
whenever the operators are regular (in the meromorphic family).
\end{lemma}

\begin{proof}

The unipotent radicals are the same for $M^{\#}$ and $M$, and the pinned representatives $\dot{w}$ lie in $\SL_2^3(F)\subset M$. Thus the intertwining integrals defined for $M^{\#}$ and  $M$ are literally the same after restriction, i.e.  
\begin{equation}\label{eq:standard-intertwining-restriction}
\operatorname{res}_{w\bm{\chi}^{\#}}\circ J(w,\bm{\chi}^{\#})
=J(w,\bm{\chi})\circ\operatorname{res}_{\bm{\chi}^{\#}}.
\end{equation}
The same observation for the longest Weyl element and the Jacquet integral gives
\begin{equation}\label{eq:jacquet-integral-restriction}
\ell_{\bm{\chi}^{\#}}^{\psi}
=\ell_{\bm{\chi}}^{\psi}\circ\operatorname{res}_{\bm{\chi}^{\#}},
\qquad
\ell_{w\bm{\chi}^{\#}}^{\psi}
=\ell_{w\bm{\chi}}^{\psi}\circ\operatorname{res}_{w\bm{\chi}^{\#}},
\end{equation}
so that $C^{\psi}(w,\bm{\chi}^{\#})=C^{\psi}(w,\bm{\chi})$ by the definition of local coefficients and \eqref{eq:standard-intertwining-restriction} \eqref{eq:jacquet-integral-restriction}. Dividing \eqref{eq:standard-intertwining-restriction} by $C^{\psi}(w,\bm{\chi})$ gives \eqref{eq:normalized-intertwining-restriction}. 
\end{proof}

In particular, for $w\in W_{\bm{\chi}}$ with its image $r$ in $R_{\bm{\chi}}$, we have $w\bm{\chi}=\bm{\chi}$ and $R(w,\bm{\chi})=R(r,\bm{\chi})$. Since $w\bm{\chi}^{\#}$ is again an extension of $\bm{\chi}$, we obtain the commutative diagram
\[
\begin{tikzcd}[column sep=huge,row sep=large]
I(\bm{\chi}^{\#})
\arrow[r,"{R(w,\bm{\chi}^{\#})}"]
\arrow[d,"\operatorname{res}_{\bm{\chi}^{\#}}"']
& I(w\bm{\chi}^{\#})
\arrow[d,"\operatorname{res}_{w\bm{\chi}^{\#}}"]\\
I(\bm{\chi})
\arrow[r,"{R(r,\bm{\chi})}"']
& I(\bm{\chi}),
\end{tikzcd}
\]
where both vertical arrows are $M$-isomorphisms.
% ===== End 02_repn of M.tex =====

\section{Statement of the results}\label{sec:results}
% ===== Begin 03_results.tex =====
In this section, we assume $\bm{\chi}^{\#}$ to be unitary, and $\bm{\pi}^{\#} = I(\bm{\chi}^{\#})$ with its decomposition 
\[
\bm{\pi}^{\#}|_M \simeq \bigoplus_{\kappa\in \Irr(R_{\bm{\chi}})} \bm{\sigma}_{\kappa}.
\]
We also assume that the central character of $\bm{\pi}^{\#}$ is trivial on the diagonal $Z^{\Delta}$, so that
\[\dim \Hom_{G^\Delta}(\bm{\pi}^{\#},\mathbf{1})=1\]
by \cite[Theorems 1.1 and 1.2]{prasad}.

\subsection{The order 4 case}\label{subsec:order-four}
If $|R_{\bm{\chi}}|=4$, then by Corollary \ref{corollary:order-4-canonical-form} we can replace $\bm{\chi}^{\#}$ by a ``canonical lift'' 
$$\bm{\chi}^{\#} = (\rho,\rho^{-1};\rho,\rho^{-1};\rho,\rho^{-1})$$ for some character $\rho:F^\times \longrightarrow \CC^1$ of order 4, without changing the restriction $I(\bm{\chi})$, $\bm{\chi}:=\bm{\chi}^{\#}|_{T_M}$. Then $\bm{\pi}^{\#} = \pi_1\boxtimes \pi_2\boxtimes \pi_3$ with each $\pi_i=\rho \times \rho^{-1}$, and we put $\xi:=\rho^2$. The purpose of this subsection is to prove the following theorem.

\begin{theorem}\label{theorem:order-four-distinction}
    In the above setting, we have $$\dim \Hom_{G^\Delta}(\bm{\sigma}_{\mathbf{1}},\mathbf{1})=1.$$ In other words, the $G^\Delta$-distinguished member in the order 4 L-packet $\{\bm{\sigma}_{\kappa}\}_{\kappa\in \Irr(R_{\bm{\chi}})}$ is the unique $\mathfrak{w}_M$-generic one.
\end{theorem}

For $\tau\in S_3$, we define the permutation map
\[
P_\tau: \bm{\pi}^{\#} = (\rho\times \rho^{-1})^{\boxtimes 3}\longrightarrow (\rho\times \rho^{-1})^{\boxtimes 3}= \bm{\pi}^{\#},\qquad v_1\otimes v_2\otimes v_3\longmapsto v_{\tau^{-1}(1)}\otimes v_{\tau^{-1}(2)}\otimes v_{\tau^{-1}(3)},
\]

which is a linear bijection with inverse $P_{\tau^{-1}}$ and satisfies $P_\tau P_{\tau'}=P_{\tau\tau'}$ for $\tau,\tau'\in S_3$. Since in our current case $\pi_1= \pi_2 = \pi_3\simeq  \rho\times \rho^{-1}$, the map $P_\tau$ is clearly a $G^\Delta$-equivariant isomorphism. We take a nonzero trilinear form 
\[
0\neq L \in \Hom_{G^\Delta}(\bm{\pi}^{\#}, \mathbf{1}),
\]
then by multiplicity one there is a unique constant $c_\tau\in \CC^\times$ such that $L\circ P_\tau = c_\tau L$, and 
\[
c: S_3\longrightarrow \CC^\times,\qquad \tau\longmapsto c_\tau
\]
is clearly a character of the finite group $S_3$. Recall that 
\[
X_M(\bm{\pi}^{\#}) = \{\mathbf{1},(\mathbf{1},\xi,\xi),(\xi,\mathbf{1},\xi),(\xi,\xi,\mathbf{1})\}\simeq (\ZZ/2\ZZ)^2.
\]
We let $S_3$ act on $X_M(\bm{\pi}^{\#})$ by the left permutation, namely 
\[
\tau\cdot(\eta_1,\eta_2,\eta_3):=(\eta_{\tau^{-1}(1)},\eta_{\tau^{-1}(2)},\eta_{\tau^{-1}(3)}),\qquad\tau\in S_3,~(\eta_1,\eta_2,\eta_3)\in X_M(\bm{\pi}^{\#}),
\]
and transport this action to $R_{\bm{\chi}}$ via the isomorphism $R_{\bm\chi}\xrightarrow{\sim}X_M(\bm\pi^{\#})$ in Proposition \ref{proposition:r-group-self-twists}. Then we define the dual action of $S_3$ on $\Irr(R_{\bm\chi})$ by
\[
\tau(\kappa)(r):=\kappa(\tau^{-1}\cdot r),\qquad \tau\in S_3,~r\in R_{\bm\chi},~\kappa\in \Irr(R_{{\bm{\chi}}}).
\]

\begin{lemma}\label{lemma:order-four-permutation-operators}
   In the above setting, we have 
    \[
    P_\tau(\bm{\sigma}_\kappa)=\bm{\sigma}_{\tau(\kappa)},\qquad \tau\in S_3,~\kappa\in \Irr(R_{{\bm{\chi}}})
    \]
\end{lemma}
\begin{proof}

Recall that for $r\in R_{\bm\chi}$, we denote by $R(r,\bm\chi)$ the corresponding normalized self-intertwining operator. The factorization \eqref{eq:normalized-intertwining-tensor-factorization}, together with Lemma \ref{lemma:normalized-intertwining-restriction}, gives
\[
P_\tau \circ R(r,\bm\chi)\circ P_\tau^{-1}=R(\tau\cdot r,\bm\chi).
\]
For $\kappa\in \Irr(R_{\bm\chi})$ and
\[
e_\kappa=\frac{1}{|R_{\bm\chi}|}\sum_{r\in R_{\bm\chi}}\kappa(r)R(r,\bm\chi),
\]
then reindexing the sum by $r\mapsto\tau\cdot r$ yields $P_\tau \circ e_\kappa\circ  P_\tau^{-1}=e_{\tau(\kappa)}$.  Since $\bm{\sigma}_\kappa=e_\kappa I(\bm\chi)$, the assertion follows.
\end{proof}

\begin{proof}[Proof of Theorem~\ref{theorem:order-four-distinction}]
    For $\kappa\in\Irr(R_{\bm\chi})$, we write $L_\kappa:=L|_{\bm{\sigma}_\kappa}$.  Since $P_\tau(\bm{\sigma}_\kappa)=\bm{\sigma}_{\tau(\kappa)}$ and $L\circ P_\tau=c_\tau L$ with $c_\tau\ne0$, we have
    \[
    L_\kappa\ne0\quad\Longleftrightarrow\quad L_{\tau(\kappa)}\ne0.
    \]
   By $\dim\Hom_{G^\Delta}(\bm\pi^{\#},\mathbf{1})=1$ and $\bm\pi^{\#}|_M$ is multiplicity free, exactly one restriction $L_\kappa$ is nonzero, and its label is therefore fixed by $S_3$.  Under the isomorphism $R_{\bm\chi}\simeq(\mathbb Z/2\mathbb Z)^2$, the group $S_3$ permutes the three nontrivial elements transitively. Thus the trivial character is the only $S_3$-fixed element of $\Irr(R_{\bm\chi})$ and we get $\kappa=\mathbf{1}$.
\end{proof}

\subsection{The order 2 case: statements}\label{subsec:order-two}

Suppose $|X_M(\bm{\pi}^{\#})|=2$, and we write 
\[
\bm{\pi}^{\#}|_M \simeq \bm{\sigma}_+ \oplus \bm{\sigma}_-,
\]
where we recall that $\bm{\sigma}_+\in \Irr(M)$ is the unique $\mathfrak{w}_M$-generic member. 

\begin{theorem}\label{theorem:order-2-main-theorem}
    Notations as above.
    \begin{enumerate}
        \item Suppose $\{\bm{\sigma}_+,\bm{\sigma}_-\}$ is of type I, and up to permutation $\xi_1 = \xi_2 = \xi$ for some character $\xi:F^\times \longrightarrow \CC^\times$ of order 2 and $\xi_3 \neq \xi$. We set $s:=\xi(-1)$, then
        \[
        \dim \Hom_{G^\Delta}(\bm{\sigma}_s, \mathbf{1}) = 1.
        \]
        \item\label{item:order-two-distinction-type-ii} Suppose $\{\bm{\sigma}_+,\bm{\sigma}_-\}$ is of type II, so that $\xi_1,\xi_2,\xi_3$ are all of order 2 and mutually distinct, with $\xi_1\xi_2\xi_3=\mathbf{1}$. In this case we set $Y:=\{\mathbf{1},\xi_1,\xi_2,\xi_3\}$ and
        \[
        s:=\prod_{\eta\in Y}\varepsilon(\frac{1}{2}, \eta,\psi)\in \{\pm 1\},
        \]
       then we have
        \[
        \dim \Hom_{G^\Delta}(\bm{\sigma}_s, \mathbf{1}) = 1.
        \]
    \end{enumerate}
\end{theorem}
Here we recall that for a general character $\chi:F^\times \longrightarrow \CC^\times$, we have the functional equation
\[ \varepsilon(s,\chi,\psi)\varepsilon(1-s,\chi^{-1},\psi) = \chi(-1), \]
so if $\chi^2= \mathbf{1}$, then $\varepsilon(1/2,\chi,\psi)^2 = \chi(-1)$. In particular, we do have
\[
        \left(\prod_{\eta\in Y}\varepsilon(\frac{1}{2}, \eta,\psi)\right)^2=\prod_{\eta\in Y}\eta(-1)=(\xi_1\xi_2\xi_3)(-1)=1
        \]
        and the $s$ in Part (\ref{item:order-two-distinction-type-ii}) really is a sign $\pm 1$. The proofs are given in the next two sections. 

\subsection{Supplements}\label{subsec:supplements}

In this subsection, we include several additional remarks. The first is the compatibility with D. Prasad's theorem in the unramified case.

\begin{lemma}\label{lemma:unramified-distinguished-constituent}
Suppose all $\mu_i,\nu_i$ are unitary unramified, then the unique ($K_M$-)unramified member of the L-packet $\{\bm{\sigma}_{\kappa}\}$ is $\bm\sigma_{\mathbf{1}}$.
\end{lemma}
\begin{proof}
Let $f^{\circ}\in I(\bm{\chi})^{K_M}$ be the spherical vector normalized by $f^\circ(1)=1$, and assume that $\psi$ is unramified. Then the Jacquet integral gives
\[
\ell_{\bm\chi}^{\psi}(f^\circ)
=\prod_{i=1}^3\bigl(1-q_F^{-1}\xi_i(\varpi_F)\bigr)\ne0
\]
since $\xi_i$ is unitary, so that the Whittaker functional $\ell_{\bm{\chi}}^{\psi}$ does not vanish on the unique unramified constituent of $I(\bm{\chi})$. 
\end{proof}

D. Prasad's theorem for unramified principal series \cite[Theorem 1.3]{prasad} states that a trilinear form on an unramified principal series of $\GL_2^3(F)$ does not vanish on the line of spherical vectors. Since there is only one unitary unramified character $\xi_0$ of order 2, a unitary unramified principal series packet $\{\bm{\sigma}_{\kappa}\}$ with more than one member belongs to the case of order 4 or order 2, type I. Thus the above lemma shows that in this case the $G^{\Delta}$-distinguished constituent is exactly the unramified one, which is compatible with Prasad's result.

Our second remark is the self-duality of our unitary principal series L-packet.

\begin{proposition}\label{proposition:unitary-principal-series-self-duality}
Let $\bm\chi:T_M\longrightarrow\CC^1$ be a unitary character trivial on $Z^\Delta$. Then each constituent $\bm\sigma_\kappa$ of $I(\bm\chi)$ is self-dual.
\end{proposition}

\begin{proof}
Choose a unitary lift $\bm\chi^\#$ and write $\pi_i=\mu_i\times\nu_i$, $\omega_i=\mu_i\nu_i$. The condition $\bm{\chi} |_{Z^{\Delta}} = \mathbf{1}$ is paraphrased as $\prod_i\omega_i=\mathbf{1}$. Since $\pi_i^\vee\simeq\pi_i\otimes\omega_i^{-1}$, we find 
 $$I(\bm{\chi}^{\#})^\vee|_M \simeq (\pi_1\boxtimes \pi_2\boxtimes \pi_3)\otimes (\omega_1\boxtimes \omega_2 \boxtimes \omega_3)^{-1}|_M \simeq I(\bm{\chi}^{\#})|_M,$$
and the L-packet $\{\bm{\sigma}_{\kappa}\}$ consisting of constituents of $I(\bm \chi)$ does not change after taking the contragredient. The unitarity implies that $\bm\sigma_{\mathbf{1}}^\vee$ is $(B_M,\psi_{U_M}^{-1})$-generic, while this is just $\mathfrak{w}_M$-generic (see the Remark after \eqref{eq:standard-whittaker-datum}), thus $\bm\sigma_{\mathbf{1}}^\vee\simeq\bm\sigma_{\mathbf{1}}$. By the work of Tadi\'c (see \S~\ref{subsec:tadic}) we know that each $\bm{\sigma}_{\kappa}$ can be written as $ \bm{\sigma}_{\kappa}\simeq \bm{\sigma}_{\mathbf{1}}^g$ for some $g\in M^{\#}$ (recall $\bm{\sigma}_{\mathbf{1}}^g(h):=\bm{\sigma}_{\mathbf{1}}(ghg^{-1})$). Since $(\bm{\sigma}_{\mathbf{1}}^g)^{\vee}\simeq (\bm{\sigma}_{\mathbf{1}}^{\vee})^g$, the proof is completed.
\end{proof}

Finally, let us record the $S_3$-action on an order $4$ packet. We let $S_3$ act on $M$ by the standard permutation 
\[
\alpha_\tau(g_1,g_2,g_3)
:=(g_{\tau^{-1}(1)},g_{\tau^{-1}(2)},g_{\tau^{-1}(3)}),
\qquad \tau \in S_3.
\]
Note that this $S_3$-action preserves the standard Whittaker datum $\mathfrak w_M$. Then this action on $M$ induces an $S_3$-action on $\Irr(M)$ by 
$$ {}^\tau \bm{\sigma}(g):=\bm{\sigma}(\alpha_{\tau^{-1}}(g)),\qquad \bm{\sigma}\in \Irr(M),~\tau\in S_3,~g\in M.$$

Suppose $\bm{\chi}:T_M\longrightarrow \CC^1$ is a unitary character with $|R_{\bm\chi}|=4$, and we have identified $\Irr(R_{\bm\chi})$ with
\[
\{(\varepsilon_1,\varepsilon_2,\varepsilon_3)\in\{\pm1\}^3\mid 
\varepsilon_1\varepsilon_2\varepsilon_3=1\}
\]
in a canonical way. We also let $S_3$ act on $\Irr(R_{\bm\chi})$ via 
\[ \tau\cdot (\varepsilon_1,\varepsilon_2,\varepsilon_3) = (\varepsilon_{\tau^{-1}(1)},\varepsilon_{\tau^{-1}(2)},\varepsilon_{\tau^{-1}(3)}),\]
then we have the following. 

\begin{proposition}\label{proposition:order-four-packet-permutation-action}
In the order 4 case, we have ${}^\tau\bm\sigma_\kappa\simeq\bm\sigma_{\tau(\kappa)}$ for $\tau\in S_3$ and $\kappa\in\Irr(R_{\bm\chi})$. In particular, $S_3$ fixes the generic member and permutes the other three representations transitively.
\end{proposition}
\begin{proof}

Recall that in the order 4 case we can pick $\bm{\pi}^{\#} = (\rho\times \rho^{-1})\boxtimes (\rho\times \rho^{-1})\boxtimes (\rho\times \rho^{-1})$ such that $\{\bm{\sigma}_{\kappa}\}$ is the set of constituents of $\bm{\pi}^{\#}|_M$. The permutation operator
\[
P_\tau: \bm{\pi}^{\#} = (\rho\times \rho^{-1})^{\boxtimes 3}\longrightarrow (\rho\times \rho^{-1})^{\boxtimes 3}= \bm{\pi}^{\#},\qquad v_1\otimes v_2\otimes v_3\longmapsto v_{\tau^{-1}(1)}\otimes v_{\tau^{-1}(2)}\otimes v_{\tau^{-1}(3)}
\]
defined in \S~\ref{subsec:order-four} clearly satisfies
\[
P_\tau\circ \bm\pi^{\#}(g)\circ P_\tau^{-1}
=\bm\pi^{\#}(\alpha_\tau(g)),\qquad g\in M.
\]
Then Lemma \ref{lemma:order-four-permutation-operators} shows that $P_\tau$ restricts to an $M$-isomorphism from ${}^\tau\bm\sigma_\kappa$ to $\bm\sigma_{\tau(\kappa)}$. 
\end{proof}

Of course, this is just a specific case of the general compatibility of local Langlands correspondence with outer automorphisms. It is not very difficult to formulate and prove a general statement for the $S_3$-action on the refined tempered L-parameters of $M$, but we shall not do it here since it is not needed in this paper. Let us just mention: we have the outer automorphism group isomorphisms 
$$\mathrm{Out}(\bM)\simeq \ZZ/2\ZZ\times S_3,\qquad \mathrm{Out}(\bM/\bZ^{\Delta})\simeq S_3.$$ 
but $\mathrm{Out}(\bM^{\#})$ is infinite (with a natural subgroup $(\ZZ/2\ZZ)^3\rtimes S_3$).
% ===== End 03_results.tex =====

\section{Proof of the order 2, type I case}\label{sec:type-i}
% ===== Begin 04_type 1.tex =====
In this section we complete a proof of the order 2, type I case. By Corollary \ref{corollary:order-2-canonical-form}, up to permutation we may assume
\[
\bm{\chi}^{\#}:=(\xi,\mathbf{1};\xi,\mathbf{1};\alpha,\alpha^{-1}),
\qquad
\bm{\chi}:=\left.\bm{\chi}^{\#}\right|_{T_M},
\]
where the character $\alpha:F^\times\longrightarrow\CC^1$ satisfies $\alpha^2\neq \xi$. We have $R_{\bm{\chi}}=\{1,r_{12}\}$, where $r_{12}$ is the image of $w_{12}$. To simplify notations, in this section we write
\begin{equation}\label{eq:type-i-rank-one-models}
(\pi_\eta,V_\eta):=(I_{\eta,\mathbf{1}},V_{\eta,\mathbf{1}}),
\qquad
(\tau_\eta,W_{\eta}):=(I_{\eta,\eta^{-1}},V_{\eta,\eta^{-1}})
\end{equation}
for a unitary character $\eta:F^\times \longrightarrow\CC^1$, and denote the underlying space $I(\bm{\chi}^{\#})$ by 
\begin{equation}\label{eq:type-i-tensor-model}X_{\xi,\alpha}:=V_\xi\otimes V_\xi\otimes W_{\alpha}.\end{equation}
Recall that the restriction of functions gives a canonical $M$-isomorphism
\[
\operatorname{res}_{\bm{\chi}^{\#}}:X_{\xi,\alpha}=I(\bm{\chi}^{\#})\xrightarrow{\sim}I(\bm{\chi}).
\] Fix a nonzero
$L\in\Hom_{G^\Delta}(I(\bm\chi),\mathbf{1})$. Our goal is to prove the identity
\begin{equation}\label{eq:type-i-period-eigenvalue}
L\circ R(r_{12},\bm{\chi})=\xi(-1)L,
\end{equation}
which then gives the required result.

\subsection{Model transition}\label{subsec:type-i-models}

We want to ``lift'' the computation of $L\circ R(r_{12},\bm{\chi})$ for $L\in\Hom_{G^\Delta}(I(\bm{\chi}),\mathbf{1})$ to the level of $I(\bm{\chi}^{\#})$ and $L^{\#}\in \Hom_{G^{\Delta}}(I(\bm{\chi}^{\#}),\mathbf{1})$, then the intertwining operator $R(w_{12},\bm{\chi}^{\#})$ can be easily specified on each $\GL_2$-factor. However, a natural problem here is: the target of $R(w_{12},\bm{\chi}^{\#})$ is $I(w_{12}\bm{\chi}^{\#})$; we have to use another restriction map
\[\mathrm{res}_{w_{12}\bm{\chi}^{\#}}:I(w_{12}\bm{\chi}^{\#})\longrightarrow I(\bm{\chi})\]
and it does not make sense to write $L^{\#}\circ R(w_{12},\bm{\chi}^{\#})$ as a functional on $I(\bm{\chi})$. 
We can easily resolve this technical issue by a model transition, in terms of determinant twist.

For characters $\mu,\nu,\eta:F^\times\longrightarrow\CC^\times$, we define the determinant twist map
\begin{equation}\label{eq:determinant-twist}
D_\eta:V_{\mu,\nu}\longrightarrow V_{\mu\eta,\nu\eta},
\qquad (D_\eta f)(g):=\eta(\det g)f(g).
\end{equation}
Here for notational convenience we deliberately suppress the source and target from the notation and the relevant spaces will be clear from the context. We then have
\begin{equation}\label{eq:determinant-twist-covariance}
I_{\mu\eta,\nu\eta}(g)\circ D_\eta
=\eta(\det g)D_\eta\circ I_{\mu,\nu}(g),
\qquad g\in G,
\end{equation}
i.e. $D_\eta$ gives an isomorphism between the twist $(I_{\mu,\nu}\otimes \eta, V_{\mu,\nu})$ and $(I_{\eta\mu,\eta\nu}, V_{\eta\mu,\eta\nu})$. We also have $D_{\eta^{-1}}\circ D_{\eta} = \mathrm{id}$, and in particular $D_\xi^2=\operatorname{id}$ since $\xi^2=\mathbf{1}$.

Let
\[
R_{\xi}:=R_{\xi,\mathbf{1}}:V_\xi = V_{\xi,\mathbf{1}}\longrightarrow V_{\mathbf{1},\xi}
\]
be the normalized intertwining operator defined in \eqref{eq:rank-one-normalized-intertwining}. We then set
\begin{equation}\label{eq:definition-active-relation-involution}
A_\xi^{\psi}:=D_{\xi}\circ R_{\xi}:V_\xi\longrightarrow V_\xi.
\end{equation}
The relation \eqref{eq:determinant-twist-covariance} now gives
\begin{equation}\label{eq:active-relation-covariance}
A_\xi^{\psi} \circ \pi_\xi(g)
=\xi(\det g)\pi_\xi(g)\circ A_\xi^{\psi},
\qquad g\in G.
\end{equation}

\begin{lemma}\label{lemma:active-relation-involution}
We have
\begin{equation}\label{eq:active-relation-involution}
\ell_{\xi,\mathbf{1}}^{\psi}\circ D_{\xi}
=\ell_{\mathbf{1},\xi}^{\psi},\qquad
\ell_{\xi,\mathbf{1}}^{\psi}\circ A_\xi^{\psi}
=\ell_{\xi,\mathbf{1}}^{\psi},\qquad
(A_\xi^{\psi})^2=\operatorname{id}_{V_\xi}.
\end{equation}
\end{lemma}

\begin{proof}
For $f\in V_{\mathbf{1},\xi}$, the definition of the Jacquet integral gives
\[
\begin{aligned}
\ell_{\xi,\mathbf{1}}^{\psi}\circ D_\xi (f)
&=\int_F (D_\xi f)(\dot{w}_0^{-1}u(x))\psi(-x)\,dx\\
&=\int_F \xi\bigl(\det(\dot{w}_0^{-1}u(x))\bigr)
f(\dot{w}_0^{-1}u(x))\psi(-x)\,dx\\
&=\int_F f(\dot{w}_0^{-1}u(x))\psi(-x)\,dx =\ell_{\mathbf{1},\xi}^{\psi}(f),
\end{aligned}
\]
in the sense of holomorphic continuation, by $\det(\dot{w}_0^{-1}u(x))=1$. The Whittaker normalization of $R_\xi=R_{\xi,\mathbf{1}}$ reads $\ell_{\mathbf{1},\xi}^{\psi}\circ R_\xi
=\ell_{\xi,\mathbf{1}}^{\psi}$, thus
\[
\ell_{\xi,\mathbf{1}}^{\psi}\circ A_\xi^\psi
=\ell_{\xi,\mathbf{1}}^{\psi}\circ D_\xi\circ R_\xi
=\ell_{\mathbf{1},\xi}^{\psi}\circ R_\xi
=\ell_{\xi,\mathbf{1}}^{\psi}.
\]
By \eqref{eq:active-relation-covariance} and $\xi^2=\mathbf{1}$ we find that $(A_\xi^{\psi})^2\in \End_G(\pi_\xi)$, hence $(A_\xi^{\psi})^2$ is a scalar by irreducibility of $\pi_\xi$. Finally, by $\ell_{\xi,\mathbf{1}}^{\psi}\circ(A_\xi^\psi)^2
=\ell_{\xi,\mathbf{1}}^{\psi}$ it is the identity.
\end{proof}

Now we write
\[
A_1:= A_\xi^{\psi}\otimes\operatorname{id}_{V_\xi}\otimes\operatorname{id}_{W_{\alpha}},
\qquad
A_2:=\operatorname{id}_{V_\xi}\otimes A_\xi^{\psi}\otimes\operatorname{id}_{W_{\alpha}},
\]
so that 
\[A_1 A_2 = A_2 A_1 = A_\xi^{\psi}\otimes A_\xi^{\psi}\otimes \operatorname{id}_{W_{\alpha}}.\]
The following Lemma provides our ``model transition''; it is clear from the proof that the order 2 condition of $\xi$ is crucial.

\begin{lemma}\label{lemma:type-i-fixed-tensor-model}
The following diagram commutes:
\[
\begin{tikzcd}[column sep=huge,row sep=large]
X_{\xi,\alpha}
\arrow[r,"{A_1A_2}"]
\arrow[d,"\sim","{\operatorname{res}_{\bm{\chi}^{\#}}}"']
& X_{\xi,\alpha}
\arrow[d,"\sim"',"{\operatorname{res}_{\bm{\chi}^{\#}}}"]\\
I(\bm{\chi})
\arrow[r,"{R(r_{12},\bm{\chi})}"']
& I(\bm{\chi}).
\end{tikzcd}
\]
In particular, the map $A_1A_2$ is $M$-equivariant.
\end{lemma}

\begin{proof}
Put
\[
X_{\xi,\alpha}^w:=I(w_{12}\bm{\chi}^{\#})
=V_{\mathbf{1},\xi}\otimes V_{\mathbf{1},\xi}\otimes W_{\alpha}
\]
and set
\[
D_{12}:=D_\xi\otimes D_\xi\otimes\operatorname{id}_{W_\alpha}:
X_{\xi,\alpha}\longrightarrow X_{\xi,\alpha}^w.
\]
By $D_\xi^2=\operatorname{id}$, we have
\begin{equation}\label{eq:type-i-factorization-before-closing}
R(w_{12},\bm{\chi}^{\#})
=R_{\xi}\otimes R_{\xi}\otimes\operatorname{id}_{W_{\alpha}}
=D_{12}\circ (A_1A_2).
\end{equation}
For $m=(g_1,g_2,g_3)\in M$ with $d=\det g_1=\det g_2=\det g_3$, from
\[
(D_{12}f)(m)=\xi(d)^2f(m)=f(m),
\]
the following diagram commutes:
\begin{equation}
    \label{eq:type-i-tensor-restriction-diagram} 
    \begin{tikzcd}
        X_{\xi,\alpha}=I(\bm{\chi}^{\#})
        \arrow[dr, "{\operatorname{res}_{\bm{\chi}^{\#}}}"']
        \arrow[rr,"{ D_{12}}"]
        && I(w_{12}\bm{\chi}^{\#})=X_{\xi,\alpha}^{w}
        \arrow[ld, "{ \operatorname{res}_{w_{12}\bm{\chi}^{\#}}}"]\\
        & I(\bm{\chi}) &
    \end{tikzcd}
\end{equation}
Combining \eqref{eq:type-i-factorization-before-closing},
\eqref{eq:type-i-tensor-restriction-diagram}, and the compatibility relatioin
\eqref{eq:normalized-intertwining-restriction}, we obtain
\[
{ \operatorname{res}_{\bm{\chi}^{\#}}\circ (A_1A_2)
=R(r_{12},\bm{\chi})\circ\operatorname{res}_{\bm{\chi}^{\#}},
}
\]
which is the asserted diagram.  Since both vertical arrows
are $M$-isomorphisms and the bottom arrow is $M$-equivariant, $A_1A_2$ is
$M$-equivariant as well.
\end{proof}

\subsection{The open-orbit flip}\label{subsec:open-orbit-flip}

Let $\gamma:F^\times\longrightarrow\CC^1$ be a unitary character, and
\[
P_{12}:X_{\xi,\gamma}\longrightarrow X_{\xi,\gamma},
\qquad
P_{12}(v_1\otimes v_2\otimes v_3):=v_2\otimes v_1\otimes v_3
\]
the 1-2 component flip map on $X_{\xi,\gamma} = V_\xi \otimes V_\xi \otimes W_\gamma$, which is a $G^{\Delta}$-isomorphism. The main result of this subsection is the following computation. 

\begin{theorem} \label{theorem:type-i-open-orbit-flip}
Let $L_\gamma^{\#}\in\Hom_{G^\Delta}(X_{\xi,\gamma},\mathbf{1})$ be a nonzero trilinear form on $X_{\xi,\gamma}$, then we have
\begin{equation}\label{eq:type-i-general-flip-scalar}
L_\gamma^{\#}\circ P_{12}=\gamma(-1)L_\gamma^{\#}.
\end{equation}
\end{theorem}

The computation starts from the Mackey filtration established by Prasad in the proof of his unramified test vector theorem \cite[Theorem 5.10]{prasad}, so we first recall his results. Temporarily, for $1\leq i\leq 3$ let $\chi_i:T\longrightarrow \CC^\times$ be a character of $T$ and $(I_B^G(\chi_i),V_i)$ the normalized parabolic induction obtained from $\chi_i$. We have a canonical isomorphism 
\[\Hom_{G^\Delta}(V_1\otimes V_2\otimes V_3,\CC)\simeq \Hom_G(V_1\otimes V_2,V_3^\vee),\]
which leads us to study the $G$-module $(V_1\otimes V_2)|_{G}$ via the Mackey theory. Identifying $B\backslash G$ with $\PP^1$ in the natural way, the double coset space
\[(B\times B)\backslash (G\times G)/G^{\Delta}\simeq (\PP^1\times \PP^1)/G^{\Delta}\]
has two elements:
\begin{enumerate}
    \item the closed orbit $\Delta =\{(x,x)\mid x\in \PP^1\}$, with the stabilizer of $(x_0,x_0)$ given by $B\subset G^{\Delta}$;
    \item the open orbit $\Omega=\{(x,y)\in \PP^1\times \PP^1\mid x\neq y\}$, with the point $(x_0,x_0s_0)=(B,Bs_0)$ having stabilizer $B\cap (s_0 B s_0^{-1}) = T\subset G^\Delta$.
\end{enumerate}
Here $x_0\in \PP^1\simeq (B\backslash G)$ corresponds to the base point $B \cdot 1\in B\backslash G$, and we recall
\[
s_0:=\begin{bmatrix}0&1\\1&0\end{bmatrix}.
\]
We have a $G$-equivariant identification of the open orbit
\[
T\backslash G\xrightarrow{\sim}\Omega,\qquad
Tg\longmapsto(Bg,Bs_0g).
\]

With these preparations, we can obtain a $G$-module exact sequence (\cite[pp. 19]{prasad})
\begin{equation}\label{eq:prasad-mackey-filtration} 0 \longrightarrow \mathrm{ind}_T^G(\chi_1\chi_2^w)\overset{\mathrm{ext}}{\longrightarrow} V_1\otimes V_2\overset{\mathrm{res}_\Delta}{\longrightarrow} I_B^G(\chi_1\chi_2\delta_B^{1/2})\longrightarrow 0,\end{equation}
where $\chi_i^w(t):=\chi_i(s_0ts_0^{-1})$, and the two maps $\mathrm{ext},\mathrm{res}_\Delta$ are defined as follows. 
\begin{itemize}
    \item The map
    \begin{equation}\label{eq:mackey-restriction} \mathrm{res}_\Delta:V_1\otimes V_2\longrightarrow I_B^G(\chi_1\chi_2\delta_B^{1/2}),\qquad F\longmapsto \bigl(g\mapsto F(g,g)\bigr)\end{equation}
    is the restriction to the closed orbit. 
    \item The map 
    \begin{equation}\label{eq:mackey-extension}\mathrm{ext}: \mathrm{ind}_T^G(\chi_1\chi_2^w)\longrightarrow V_1\otimes V_2\end{equation}
    is extension by $0$ to the closed orbit. More precisely, we identify the open orbit $\Omega$ with $T\backslash G$, and for $f\in \mathrm{ind}_T^G(\chi_1\chi_2^w)$ define 
    \[ \mathrm{ext}(f)(g,g):=0,\qquad \mathrm{ext}(f)(g,s_0 g):=f(g),\qquad g\in G.\]
    Then we extend $\mathrm{ext}(f)$ to a function on $G\times G$ by $B\times B$-equivariance, which is well-defined.
\end{itemize}

This ends our brief review of Prasad's Mackey filtration \eqref{eq:prasad-mackey-filtration}. 

Now we apply this to our case $\chi_1 = \chi_2 = \xi \boxtimes \mathbf{1}$ and $\chi_3 = \gamma\boxtimes \gamma^{-1}$, where $\xi$ is of order 2 and $\gamma$ is unitary. Then the exact sequence \eqref{eq:prasad-mackey-filtration} becomes 
\begin{equation}\label{eq:concrete-prasad-mackey} 0 \longrightarrow \mathrm{ind}_T^G(\xi \boxtimes \xi)\overset{\mathrm{ext}}{\longrightarrow} V_\xi \otimes V_\xi\overset{\mathrm{res}_{\Delta}}{\longrightarrow} |\cdot|^{1/2}\times |\cdot|^{-1/2}\longrightarrow 0.\end{equation}
The first key observation is the following.
\begin{lemma}\label{lemma:open-orbit-flip-diagram}
Define
\[
\mathscr S:\mathrm{ind}_T^G(\xi\boxtimes\xi)
\longrightarrow
\mathrm{ind}_T^G(\xi\boxtimes\xi),
\qquad
(\mathscr Sf)(g):=f(s_0g),
\]
and by abuse of notation we still write
\[
P_{12}:V_\xi\otimes V_\xi\longrightarrow V_\xi\otimes V_\xi,
\qquad
P_{12}(v_1\otimes v_2):=v_2\otimes v_1,
\]
for the flip map. Then the following diagram commutes.
\[
\begin{tikzcd}
\mathrm{ind}_T^G(\xi\boxtimes\xi)
\arrow[r,"\mathrm{ext}"]
\arrow[d,"\mathscr S"']
&
V_\xi\otimes V_\xi
\arrow[d,"P_{12}"]
\\
\mathrm{ind}_T^G(\xi\boxtimes\xi)
\arrow[r,"\mathrm{ext}"']
&
V_\xi\otimes V_\xi.
\end{tikzcd}
\]
\end{lemma}

\begin{proof}
It is easy to check that $\mathscr S$ is a well-defined $G$-isomorphism, by
\[
s_0\diag(a,d)s_0^{-1}=\diag(d,a),\qquad a,d\in F^\times
\]
and the $s_0$-invariance of the character $\xi\boxtimes\xi$. In the induced
model $V_\xi\otimes V_\xi$, the flip is given by
\[
(P_{12}F)(g_1,g_2)=F(g_2,g_1),\qquad g_1,g_2\in G,~F\in V_\xi\otimes V_\xi.
\]
As functions on $G\times G$, both $P_{12}\circ\mathrm{ext}(f)$ and
$\mathrm{ext}\circ\mathscr S(f)$ vanish on the closed diagonal orbit. On the
open orbit, by $s_0^2=1$ we have
\[
\bigl(P_{12}\circ\mathrm{ext}(f)\bigr)(g,s_0g)=\mathrm{ext}(f)(s_0g,g)=\mathrm{ext}(f)\bigl(s_0g,s_0(s_0g)\bigr)=f(s_0g),
\]
whereas
\[
\bigl(\mathrm{ext}\circ\mathscr S(f)\bigr)(g,s_0g)
=(\mathscr Sf)(g)
=f(s_0g).
\]
\end{proof}

We apply the functor $\Hom_G(-,W_{\gamma}^\vee)$ to this exact sequence, and from $\Hom_G(|\cdot|^{1/2}\times |\cdot|^{-1/2},W_{\gamma}^\vee)=0$ we get an embedding
\begin{equation}
    \label{eq:trilinear-form-open-orbit-embedding} \mathrm{ext}^*: \Hom_{G^\Delta}(V_\xi \otimes V_\xi \otimes W_\gamma,\CC)\simeq \Hom_G(V_\xi \otimes V_\xi, W_{\gamma}^\vee)\hookrightarrow \Hom_G(\mathrm{ind}_T^G(\xi \boxtimes \xi),W_{\gamma}^\vee).
\end{equation}
Next, the Frobenius reciprocity gives 
\begin{equation}\label{eq:frobenius-reciprocity} \mathrm{Frob}_{\gamma}:\Hom_G(\mathrm{ind}_T^G(\xi \boxtimes \xi),W_{\gamma}^\vee)\simeq \Hom_G(W_{\gamma},\mathrm{Ind}_T^G(\xi^{-1}\boxtimes \xi^{-1}))\simeq \Hom_T(W_\gamma,\xi\boxtimes \xi),\end{equation}
where we used $\xi^{-1}=\xi$ for the last isomorphism. 

Since the central character of $(\tau_\gamma,W_\gamma)$ is trivial, which coincides with the restriction of $\xi\boxtimes \xi$ to the center of $G$, by a theorem of Waldspurger \cite[Lemma 5.6 (a)]{prasad} we have 
\[\dim \Hom_T(W_\gamma,\xi\boxtimes \xi)= 1.\]
Now from the previous Lemma we then get the following key computational result. 

\begin{proposition}\label{proposition:reduction-to-waldspurger-functional} The composite $\mathrm{Frob}_\gamma\circ\operatorname{ext}^*$ gives an isomorphism 
    \[\Hom_{G^\Delta}(X_{\xi,\gamma}, \CC)\overset{\sim}{\longrightarrow} \operatorname{Hom}_T(W_\gamma,\xi\boxtimes\xi),\]
    and the following diagram commutes:
\begin{equation}\label{eq:waldspurger-functional-diagram}\begin{tikzcd} \Hom_{G^\Delta}(X_{\xi,\gamma}, \CC) \arrow[rr,"\sim"',"\mathrm{Frob}_\gamma\circ\operatorname{ext}^*"] \arrow[d,"{L_\gamma^{\#}\mapsto L_\gamma^{\#}\circ P_{12}}"'] && \operatorname{Hom}_T(W_\gamma,\xi\boxtimes\xi) \arrow[d,"{\ell\mapsto\ell\circ\tau_\gamma(s_0)}"] \\ \Hom_{G^\Delta}(X_{\xi,\gamma}, \CC) \arrow[rr,"\sim","\mathrm{Frob}_\gamma\circ\operatorname{ext}^*"'] && \operatorname{Hom}_T(W_\gamma,\xi\boxtimes\xi). \end{tikzcd}\end{equation}
In particular, for $0\neq L_\gamma^{\#}\in \Hom_{G^\Delta}(X_{\xi,\gamma},\CC)$, $0\neq \ell \in \Hom_T(W_\gamma,\xi\boxtimes \xi)$ and $c\in \CC$, we have 
\begin{equation} \label{eq:trilinear-torus-reduction}
    L_\gamma^{\#}\circ P_{12} = cL_\gamma^{\#} \Longleftrightarrow \ell\circ \tau_{\gamma}(s_0)=c\ell.
\end{equation}
\end{proposition}

\begin{proof}By \eqref{eq:trilinear-form-open-orbit-embedding}, the composite $\mathrm{Frob}_\gamma\circ\operatorname{ext}^*$ is injective. Its domain (resp. codomain) is one-dimensional by Prasad's theorem (resp. Waldspurger's theorem) recalled above, thus the composite is an isomorphism. Lemma \ref{lemma:open-orbit-flip-diagram} and the naturality of
compact Frobenius reciprocity give
\[
\begin{aligned}
\mathrm{Frob}_\gamma\circ
\operatorname{ext}^*(L_\gamma^{\#}\circ P_{12})
&=\mathrm{Frob}_\gamma
  \bigl(\operatorname{ext}^*(L_\gamma^{\#})\circ\mathscr S\bigr)\\
&=\mathrm{Frob}_\gamma\circ
\operatorname{ext}^*(L_\gamma^{\#})\circ\tau_\gamma(s_0),
\end{aligned}
\]
which gives the commutativity of the diagram. The equivalence \eqref{eq:trilinear-torus-reduction} then follows immediately.
\end{proof}

We can easily produce an explicit element (a ``Waldspurger functional") in $\operatorname{Hom}_T(W_\gamma,\xi\boxtimes\xi)$, based on the following lemma. 

\begin{lemma}\label{lemma:waldspurger-functional-formula}
    Let $\delta:F^\times \longrightarrow\CC^1$ be a unitary character, we define 
    \begin{equation}\label{eq:definition-type-i-torus-functional}
\varphi_\delta: W_{\delta}\longrightarrow \CC,\qquad \varphi_\delta(f):=\int_{F^\times}f(\overline{u}(x))\delta(x)|x|^{1/2}\,d^\times x,
\qquad
\overline{u}(x):=\begin{bmatrix}1&0\\x&1\end{bmatrix},
\end{equation}
then we have
\begin{equation}\label{eq:type-i-torus-hom-dimension}
\Hom_T(\tau_\delta,\mathbf{1})=\CC\varphi_\delta.
\end{equation}
\end{lemma}

\begin{proof}
    The formal $T$-invariance of $\varphi_\delta$ is a simple direct computation: from the matrix identity
\[
\overline{u}(x)t=t\,\overline{u}(ad^{-1}x),\qquad t=\diag(a,d)\in T,
\]
one can easily verify $\varphi_\delta\circ\tau_\delta(t)=\varphi_\delta$ 
by the substitution $x\longmapsto da^{-1}x$ in the integral.
It remains to verify the absolute convergence and nonvanishing.

We let $K:=\GL_2(\mathcal{O}_F)\subset G$ and set $C_f:=\sup_{k\in K}|f(k)|$ for our fixed element $f\in W_\delta$. For $x\in F$ with $|x|\leq 1$, $\overline{u}(x)\in K$ gives $|f(\overline{u}(x))|\leq C_f$, combined with unitarity of $\delta$ we get 
\[\left| f(\overline{u}(x))\delta(x)|x|^{1/2}\right|\leq C_f|x|^{1/2},\qquad |x|\leq 1,\]
With the fixed measure $d^\times x=dx/|x|$, put
$c_F:=\operatorname{vol}_{dx}(\mathcal O_F^\times)$, then we have
\begin{equation}
    \label{eq:torus-integrability-bound} \int_{|x|\leq 1} |x|^{1/2}d^\times x = \int_{|x|\leq 1} |x|^{-1/2}d x =c_F\sum_{n\geq 0}q_F^{-n/2}<\infty .
\end{equation}
For $x\in F$ with $|x|\geq 1$, we first use the identity
\begin{equation}\label{eq:lower-unipotent-weyl-identity}\overline{u}(x) = \begin{bmatrix}
    -x^{-1} & 1 \\ & x
\end{bmatrix}\begin{bmatrix}
1 & \\ x^{-1} & 1
\end{bmatrix} \begin{bmatrix}
    & 1 \\ 1 &
\end{bmatrix} = \begin{bmatrix}
    -x^{-1} & 1 \\ & x
\end{bmatrix}\overline{u}(x^{-1})s_0\end{equation}
and $|x^{-1}|\leq 1$ to get 
\[|f(\overline{u}(x))|=|\delta(-1)\delta(x)^{-2}|x|^{-1}f(\overline{u}(x^{-1})s_0)|\leq C_f|x|^{-1}\]
by unitarity of $\delta$. Then we have 
\[\left| f(\overline{u}(x))\delta(x)|x|^{1/2}\right|\leq C_f|x|^{-1/2},\qquad |x|\geq 1,\]
and hence $|\varphi_\delta(f)|\leq2c_FC_f\sum_{n\geq0}q_F^{-n/2}$ by the bound \eqref{eq:torus-integrability-bound}. To see that $\varphi_\delta$ is nonzero, choose a compact open subset $X\subset F^\times$ and a section $f_X\in W_\delta$, supported on the corresponding compact subset of the open Bruhat cell, so that
$f_X(\overline u(x))=\delta(x)^{-1}|x|^{-1/2}$ for $x\in X$.  Then
$\varphi_\delta(f_X)=\operatorname{vol}_{d^\times x}(X)\neq0$. 
\end{proof}

\begin{corollary}\label{corollary:torus-weyl-scalar} For $\varphi \in \Hom_T(\tau_\delta,\mathbf{1})$, we have
    \begin{equation}\label{eq:type-i-torus-weyl-scalar}
\varphi\circ\tau_\delta(s_0)=\delta(-1)\varphi.
\end{equation}
\end{corollary}

\begin{proof}
By the previous lemma, it suffices to verify this for $\varphi = \varphi_\delta$. Replacing $x$ by $x^{-1}$ in \eqref{eq:lower-unipotent-weyl-identity} and using $\delta(-1)^{-1}=\delta(-1)$ gives instead
\[
f(\overline u(x)s_0)
=\delta(-1)\delta(x)^{-2}|x|^{-1}f(\overline u(x^{-1})).
\]
Consequently, by the substitution $y=x^{-1}$ we get
\begin{align*}
(\varphi_\delta\circ\tau_\delta(s_0))(f)
& =\int_{F^\times}f(\overline u(x)s_0)\delta(x)|x|^{1/2}\,d^\times x\\
& =\delta(-1)\int_{F^\times}f(\overline u(x^{-1}))
\delta(x)^{-1}|x|^{-1/2}\,d^\times x\\
& =\delta(-1)\int_{F^\times}f(\overline u(y))
\delta(y)|y|^{1/2}\,d^\times y\\
& =\delta(-1)\varphi_\delta(f).
\end{align*}
\end{proof}

\begin{corollary}\label{corollary:twisted-torus-weyl-scalar}
    For $\ell\in \Hom_T(W_\gamma, \xi\boxtimes \xi)$, we have 
    \begin{equation}
        \label{eq:twisted-torus-weyl-scalar} \ell\circ \tau_\gamma(s_0)=\gamma(-1)\ell.
    \end{equation}
\end{corollary}

\begin{proof}
    We apply Lemma \ref{lemma:waldspurger-functional-formula} to $\delta=\xi \gamma$, which gives $0\neq \varphi_{\xi\gamma}\in \Hom_T(\tau_{\xi\gamma},\mathbf{1})$. Then we take the composite 
    \[\ell_0: W_\gamma\overset{D_\xi}{\longrightarrow} W_{\xi\gamma}\overset{\varphi_{\xi\gamma}}{\longrightarrow}\CC,\]
    where the determinant twist $D_\xi$ was introduced in \eqref{eq:determinant-twist}.
    Then \eqref{eq:determinant-twist-covariance} becomes
\begin{equation}\label{eq:concrete-det-twist}D_\xi(\tau_\gamma(g)f) = \xi(\det g)\cdot \tau_{\xi\gamma}(g)(D_\xi f),\qquad g\in G,~f\in W_\gamma
 \end{equation}
and the previous Corollary gives 
\begin{equation}\label{eq:twisted-torus-functional-weyl-action}
    \varphi_{\xi\gamma}\circ \tau_{\xi\gamma}(s_0) = \xi(-1)\gamma(-1)\varphi_{\xi\gamma}.
\end{equation}
Combined with the $T$-invariance of $\varphi_{\xi\gamma}$ and $\xi(\det t) = (\xi\boxtimes\xi)(t)$ for $t\in T$, we get 
\[\ell_0(\tau_\gamma(t)f)
=(\xi\boxtimes\xi)(t)\cdot \ell_0(f),
\qquad t\in T,\]
i.e. $\ell_0\in \Hom_T(\tau_{\gamma},\xi\boxtimes \xi)$. It suffices to prove the statement for this $\ell_0$. Taking $g = s_0$ in \eqref{eq:concrete-det-twist} we get 
\[
D_\xi\circ \tau_\gamma(s_0)=\xi(\det s_0)\cdot \tau_{\xi\gamma}(s_0)\circ D_\xi=\xi(-1)\tau_{\xi\gamma}(s_0)\circ D_\xi
.
\]
Together with \eqref{eq:twisted-torus-functional-weyl-action}, this yields
\begin{align*}
     \ell_0\circ\tau_\gamma(s_0) & = \varphi_{\xi\gamma}\circ D_\xi \circ \tau_\gamma(s_0) = \xi(-1)\varphi_{\xi\gamma}\circ \tau_{\xi\gamma}(s_0)\circ D_\xi\\ 
    & =\gamma(-1)\cdot (\varphi_{\xi\gamma}\circ D_\xi) = \gamma(-1)\cdot\ell_0.
\end{align*}
\end{proof}

Combining the above Corollary with \eqref{eq:trilinear-torus-reduction} established in Proposition \ref{proposition:reduction-to-waldspurger-functional}, we have proved Theorem \ref{theorem:type-i-open-orbit-flip}.

\subsection{Proof of the order 2, type I case}\label{subsec:type-i-proof}
We set $\beta:=\xi\alpha$ and use the notations
\[
D_\xi:W_{\beta}\longrightarrow W_{\alpha},
\qquad
X_{\xi,\beta}:=V_\xi\otimes V_\xi\otimes W_{\beta}.
\]
We define two $G^{\Delta}$-isomorphisms $T_1,T_2:X_{\xi,\beta}\longrightarrow X_{\xi,\alpha}$ by
\[
T_1:=A_\xi^\psi\otimes\operatorname{id}_{V_\xi}\otimes D_\xi,
\qquad
T_2:=\operatorname{id}_{V_\xi}\otimes A_\xi^\psi\otimes D_\xi.
\]
where $A_\xi^\psi$ is defined in \eqref{eq:definition-active-relation-involution} and $G^{\Delta}$-equivariance follows from \eqref{eq:active-relation-covariance}. Both paths of the diagram 
\begin{equation}\label{eq:type-i-flip-transition-diagram}
\begin{tikzcd}[column sep=huge,row sep=large]
X_{\xi,\beta}
\arrow[r,"{P_{12}}"]
\arrow[d,"{T_1}"']
& X_{\xi,\beta}
\arrow[d,"{T_2}"]\\
X_{\xi,\alpha}
\arrow[r,"{P_{12}}"']
& X_{\xi,\alpha}
\end{tikzcd}
\end{equation}
send $v_1\otimes v_2\otimes v_3$ to
$v_2\otimes A_\xi^\psi v_1\otimes D_\xi v_3$; hence the diagram commutes.

\begin{proposition}\label{proposition:type-i-moving-relation-involution}
For $L_{\alpha}^{\#}\in \Hom_{G^\Delta}(X_{\xi,\alpha},\mathbf{1})$, we have
\begin{equation}\label{eq:type-i-moving-identity}
L_\alpha^{\#}\circ A_1
=\xi(-1)L_\alpha^{\#}\circ A_2.
\end{equation}
Equivalently, for $v_1,v_2\in V_\xi$ and $v_3\in W_{\alpha}$,
\begin{equation}\label{eq:type-i-moving-identity-on-vectors}
L_\alpha^{\#}(A_\xi^\psi v_1\otimes v_2\otimes v_3)
=\xi(-1)L_\alpha^{\#}(v_1\otimes A_\xi^\psi v_2\otimes v_3).
\end{equation}
\end{proposition}

\begin{proof}
Suppose $L_{\alpha}^{\#}\neq 0$, and set
\[
L_\beta^{\#}:=L_\alpha^{\#}\circ T_2\in \Hom_{G^\Delta}(X_{\xi,\beta},\CC),
\]
which is nonzero since $T_2$ is a $G^\Delta$-isomorphism. Applying Theorem \ref{theorem:type-i-open-orbit-flip} to $L_\alpha^{\#}$ and $L_\beta^{\#}$, we get
\[
L_\alpha^{\#}\circ P_{12}=\alpha(-1)L_\alpha^{\#},
\qquad
L_\beta^{\#}\circ P_{12}=\beta(-1)L_\beta^{\#}.
\]
Using the above commutative diagram \eqref{eq:type-i-flip-transition-diagram}, i.e. $T_2\circ P_{12}= P_{12}\circ T_1$, we obtain
\[
\begin{aligned}
\beta(-1)L_\alpha^{\#}\circ T_2
&=\beta(-1)L_\beta^{\#}
=L_\beta^{\#}P_{12}
=L_\alpha^{\#}\circ T_2\circ P_{12}\\
&=L_\alpha^{\#}\circ P_{12} \circ T_1
=\alpha(-1)L_\alpha^{\#}\circ T_1.
\end{aligned}
\]
Since $\beta=\xi\alpha$, this is just
\begin{equation}\label{eq:type-i-comparison-with-spectator}
L_\alpha^{\#}\circ T_1=\xi(-1)L_\alpha^{\#}\circ T_2.
\end{equation}
Let $v_3\in W_\alpha$ and put
$w_3:=D_\xi v_3\in W_\beta$. Evaluating
\eqref{eq:type-i-comparison-with-spectator} on
$v_1\otimes v_2\otimes w_3$ gives
\[
L_\alpha^{\#}(A_\xi^\psi v_1\otimes v_2\otimes v_3)
=\xi(-1)L_\alpha^{\#}(v_1\otimes A_\xi^\psi v_2\otimes v_3),
\]
because $D_\xi w_3=D_\xi^2v_3=v_3$. This proves
\eqref{eq:type-i-moving-identity-on-vectors}, and hence
\eqref{eq:type-i-moving-identity}.
\end{proof}

Since $A_2^2=\operatorname{id}$ by \eqref{eq:active-relation-involution}, we obtain
\begin{equation}\label{eq:type-i-fixed-model-period-eigenvalue}
L_\alpha^{\#}\circ A_1A_2
=\xi(-1)L_\alpha^{\#}.
\end{equation}

For our fixed $0\neq L\in\Hom_{G^\Delta}(I(\bm\chi),\mathbf{1})$, we set
\[
L_\alpha^{\#}:=L\circ\operatorname{res}_{\bm\chi^{\#}}
\in\Hom_{G^\Delta}(X_{\xi,\alpha},\mathbf{1}),
\]
which is nonzero because
$\operatorname{res}_{\bm\chi^{\#}}$ is an isomorphism.  Now Lemma
\ref{lemma:type-i-fixed-tensor-model} gives
\[
(L\circ R(r_{12},\bm{\chi}))\circ\operatorname{res}_{\bm{\chi}^{\#}}
=L\circ\operatorname{res}_{\bm{\chi}^{\#}}\circ A_1A_2=L_\alpha^{\#}\circ A_1A_2
=\xi(-1)L_\alpha^{\#}.
\]
Since $\operatorname{res}_{\bm{\chi}^{\#}}$ is an isomorphism, this proves \eqref{eq:type-i-period-eigenvalue}, that is
\[L\circ R(r_{12},\bm{\chi})=\xi(-1)L.\]
We put
\[
s:=\xi(-1)\in\{\pm1\}.
\]
With the Whittaker normalization,
\[
\bm{\sigma}_+=\ker\bigl(R(r_{12},\bm{\chi})-1\bigr),
\qquad
\bm{\sigma}_-=\ker\bigl(R(r_{12},\bm{\chi})+1\bigr).
\]
If $v\in\bm{\sigma}_t$ with $t\in\{\pm1\}$, then
\[
tL(v)
=L\bigl(R(r_{12},\bm{\chi})v\bigr)
=sL(v).
\]
Thus $L$ vanishes on $\bm{\sigma}_{-s}$, and since $L\neq0$, it is nonzero on $\bm{\sigma}_s$. Consequently, we get
\[
\dim\Hom_{G^\Delta}(\bm{\sigma}_s,\mathbf{1})=1,
\qquad
\dim\Hom_{G^\Delta}(\bm{\sigma}_{-s},\mathbf{1})=0,
\]
which proves the order $2$, type I case.
% ===== End 04_type 1.tex =====

\section{Proof of the order 2, type II case}\label{sec:type-ii}
% ===== Begin 05_type 2.tex =====
In this section we match the computation and the context of Venkatesh-Wang \cite[Lemma 11.4.2]{venkatesh-wang} with our setting, and explain how a proof of the order 2, type II case follows. Let $\mu,\nu:F^\times \longrightarrow \CC^\times$ be two characters. Let us recall that the ``usual'' induced model of $I_B^G(\chi)=I_B^G(\mu\boxtimes\nu)$ is denoted by $(I_{\mu, \nu},V_{\mu,\nu})$, which consists of smooth functions $f: G\longrightarrow \CC$ such that 
\[
f\left(\begin{bmatrix}
    a& b\\ & d
\end{bmatrix}g\right)= \mu(a)\nu(d) \left|\frac{a}{d}\right|^{1/2}f(g),\qquad g\in G,~a,d\in F^\times, ~b\in F,
\]
with the right regular $G$-action 
\[
\left(I_{\mu,\nu}(g)f\right)(x):=f(xg),\qquad x,g\in G.
\]

Now let $\cH_{\mu,\nu}$ be the space of smooth functions $\varphi:F^2\setminus\{0\}\longrightarrow\CC$ such that
\[
\varphi(tz)=|t|^{-1}(\mu^{-1}\nu)(t)\varphi(z),\qquad t\in F^\times,~z\in F^2\setminus\{0\}.
\]
We regard $F^2$ as the set of row vectors, and let $G$ act on $\cH_{\mu,\nu}$ by
\[
\left(\pi_{\mu,\nu}(g)\varphi\right)(z):=\mu(\det g)|\det g|^{1/2}\varphi(zg),\qquad \varphi\in\cH_{\mu,\nu},~z\in F^2\setminus\{0\},~g\in G.
\]
Then we have a $G$-isomorphism
\begin{equation}\label{eq:homogeneous-model-isomorphism}
\iota_{\mu,\nu}:\cH_{\mu,\nu}\overset{\sim}{\longrightarrow}V_{\mu,\nu},\qquad
\varphi\longmapsto f_\varphi(g):=\mu(\det g)|\det g|^{1/2}\varphi(e_2g),
\end{equation}
where $e_2=(0,1)\in F^2$. Thus $(\pi_{\mu,\nu},\cH_{\mu,\nu})$ gives another model of $(I_{\mu,\nu},V_{\mu,\nu})$, which we shall call the ``homogeneous model''. A computational advantage of this homogeneous model is that we can pull it back to the Schwartz space $\cS(F^2)$. More precisely, we put $\xi: = \mu \nu^{-1}$, then we have a linear surjection
\[
q_{\mu,\nu}:\cS(F^2)\longrightarrow\cH_{\mu,\nu}
\]
defined by
\[
q_{\mu,\nu}(\Phi)(z):=\int_{F^\times}\Phi(tz)\xi(t)|t|\,d^\times t.
\]
where we recall $d^\times t=dt/|t|$. Note that if $\xi=\xi_0|\cdot|^s$ with $\xi_0$ unitary, this integral converges absolutely on $\operatorname{Re}(s)>-1$ (in particular for the unitary characters used below) and admits meromorphic continuation, by computation in Tate's thesis.

Let $dw=dw_1\,dw_2$ be the product measure on $F^2$, where each $dw_i$ is the self-dual Haar measure on $F$ with respect to our fixed character $\psi:F\longrightarrow\CC^1$. The ``symplectic Fourier transform'' in \cite{venkatesh-wang} is defined by
\[
\cF_\psi(\Phi)(z):=\int_{F^2}\Phi(w)\psi(z\wedge w)dw,\qquad \Phi\in\cS(F^2),~z\in F^2,
\]
where
\[
z\wedge w:=z_1w_2-z_2w_1,\qquad z=(z_1,z_2),~w=(w_1,w_2),
\]
is the standard symplectic form on $F^2$. With this convention, the Fourier inversion formula reads
\[
\cF_\psi\circ\cF_\psi=\mathrm{id}.
\]

The Fourier transform $\cF_\psi$ gives a ``lift" of the standard intertwining operator. To clarify this, we first give a formula for the (unnormalized) intertwining operator on homogeneous models. We say a smooth function $\rho:F^2\setminus\{0\}\longrightarrow\CC$ is $(-2)$-homogeneous if
\begin{equation}\label{eq:degree-minus-two-homogeneity}
\rho(tz)=|t|^{-2}\rho(z),\qquad t\in F^\times.
\end{equation}
For such a $(-2)$-homogeneous (smooth) function $\rho$, we set
\[
\int_{\PP^1(F)}\rho(z)d^\PP z:=\int_F\rho(1,x)dx.
\]
The integral converges, since $\rho$ is bounded on the two compact sets $\{(1,x)\mid |x|\leq1\}$ and $\{(x,1)\mid |x|\leq1\}$, while homogeneity gives $\rho(1,x)=|x|^{-2}\rho(x^{-1},1)$ for $|x|>1$. For such $\rho$ we have
\[
\int_{\PP^1(F)}\rho(zg)d^\PP z
=|\det g|^{-1}\int_{\PP^1(F)}\rho(z)d^\PP z,\qquad g\in G.
\]
For any $z\in F^2\setminus\{0\}$ and $\varphi\in\cH_{\mu,\nu}$, the smooth function
\[
z'\longmapsto
\frac{\xi(z'\wedge z)}{|z'\wedge z|}\varphi(z')
\]
satisfies the condition \eqref{eq:degree-minus-two-homogeneity} but is only defined on $F^2\backslash Fz$. Suppose $\xi = \xi^{\circ}|\cdot|^s$ for a unitary $\xi^{\circ}$ and $s\in \CC$, then a direct computation (of zeta integral) shows that
\begin{equation} \label{eq:homogeneous-intertwining-integral}
(M_{\mu,\nu}\varphi)(z)=\int_{\PP^1(F)}
\frac{\xi(z'\wedge z)}{|z'\wedge z|}
\varphi(z')d^\PP z'
\end{equation}
converges absolutely for all $\varphi \in \cH_{\mu,\nu}$ if and only if $\Re(s)>0$, and in this case gives a $G$-equivariant map
\[
M_{\mu,\nu}:\cH_{\mu,\nu}\longrightarrow\cH_{\nu,\mu}.
\]
Then it admits meromorphic continuation to general $\xi$. 
\begin{remark}
    Here the continuation process is the same as that for standard intertwining operators. For given $\mu,\nu$ and $\xi = \mu\nu^{-1}$, we put 
    \[ \mu_t:=\mu|\cdot|^{t/2},\qquad \nu_t:=\nu|\cdot|^{-t/2},\qquad \xi_t :=\mu_t\nu_t^{-1} = \xi |\cdot|^{t}\]
    for $t\in \CC$, then the above integral \eqref{eq:homogeneous-intertwining-integral} converges for $\Re(s+t)>0$. Then we choose a common compact model for different $\cH_{\mu_t,\nu_t}$ (or $\cH_{\nu_t,\mu_t}$), which can be the restriction of functions in $\cH_{\mu_t,\nu_t}$ to
    \[ S := \{(z_1,z_2)\in \cO_F^2\mid \max(|z_1|,|z_2|) = 1\}\subset F^2\backslash \{0\}.\]
    This model depends only on $\xi$ and not on $t$, so that the continuation of \eqref{eq:homogeneous-intertwining-integral} makes sense. 

    The point is that it suffices to carry out computations in the convergence range, and results at other regular points of characters will follow from continuation process. 
\end{remark}

Recall our notations
\[
\dot w_0=\begin{bmatrix}0&1\\-1&0\end{bmatrix},
\qquad
\dot w_0^{-1}=\begin{bmatrix}0&-1\\1&0\end{bmatrix},
\qquad
u(x):=\begin{bmatrix}1&x\\0&1\end{bmatrix},\qquad x\in F. 
\]
We denote by
\[
J_{\mu,\nu}:V_{\mu,\nu}\longrightarrow V_{\nu,\mu}
\]
the standard unnormalized intertwining operator, defined by the intertwining integral
\[
\left(J_{\mu,\nu}f\right)(g):=\int_F f(\dot w_0^{-1}u(x)g)dx
\]
in the sense of meromorphic continuation. The following is then a $\GL_2$-version of Lemma 11.3.1 of \cite[\S 11]{venkatesh-wang}, where they carry out the computation in the case $\mu\nu = \mathbf{1}$ (i.e. when the central character of $\pi_{\mu,\nu}$ is trivial and it descends to a representation of $\PGL_2(F)$).

\begin{lemma}\label{lemma:rank-one-fourier-intertwining}
As an identity of meromorphic families, we have the following commutative diagram, which specializes at every parameter where all the displayed maps are regular:
\[
\begin{tikzcd}
\cS(F^2) \arrow[d,twoheadrightarrow,"q_{\mu,\nu}"'] \arrow[rrr,"\cF_\psi"] &&&
\cS(F^2) \arrow[d,twoheadrightarrow,"q_{\nu,\mu}"]\\
\cH_{\mu,\nu}
\arrow[d,"\iota_{\mu,\nu}"',"\sim"]
\arrow[rrr,"{\gamma(1,\xi^{-1},\psi)^{-1}M_{\mu,\nu}}"']
&&&
\cH_{\nu,\mu}
\arrow[d,"\iota_{\nu,\mu}","\sim"']\\
V_{\mu,\nu}
\arrow[rrr,"{\gamma(1,\xi^{-1},\psi)^{-1}J_{\mu,\nu}}"']
&&&
V_{\nu,\mu}.
\end{tikzcd}
\]

\end{lemma}

\begin{proof}
Write $\xi=\xi^\circ|\cdot|^s$ with $\xi^\circ$ unitary. It suffices first to assume $0<\Re(s)<1$, where the defining integrals for the two sides converge absolutely; the one-dimensional Fourier identity used below is understood in the usual Tate sense. The general case follows from the uniqueness of continuation. For $\Phi\in\cS(F^2)$, we have
\[
q_{\nu,\mu}(\cF_\psi\Phi)(z)
=\int_{F^2}\Phi(w)\left(\int_F\xi^{-1}(\lambda)\psi\bigl(\lambda(z\wedge w)\bigr)d\lambda\right)dw.
\]
A direct computation (which is classical in Tate's thesis) gives
\[
\int_F\xi^{-1}(\lambda)\psi(\lambda x)d\lambda
=\xi(-1)\gamma(1,\xi^{-1},\psi)^{-1}\frac{\xi(x)}{|x|}
=\gamma(1,\xi^{-1},\psi)^{-1}\frac{\xi(-x)}{|x|}.
\]
Since $w\wedge z=-(z\wedge w)$, it follows that
\[
q_{\nu,\mu}(\cF_\psi\Phi)(z)
=\gamma(1,\xi^{-1},\psi)^{-1}
M_{\mu,\nu}\bigl(q_{\mu,\nu}(\Phi)\bigr)(z).
\]
This proves the commutativity of the upper square.

The comparison between $M_{\mu,\nu}$ and the standard intertwining operator $J_{\mu,\nu}$ is also a direct computation. Let $\varphi\in\cH_{\mu,\nu}$. Using the integration formula \eqref{eq:homogeneous-intertwining-integral} and 
\[
(z'g)\wedge(e_2g)=\det(g)(z'\wedge e_2),
\]
we obtain
\[
\begin{aligned}
(M_{\mu,\nu}\varphi)(e_2g)
&=
|\det g|
\int_{\PP^1(F)}
\frac{\xi((z'g)\wedge(e_2g))}
{|(z'g)\wedge(e_2g)|}
\varphi(z'g)d^\PP z'\\
&=
\xi(\det g)
\int_{\PP^1(F)}
\frac{\xi(z'\wedge e_2)}
{|z'\wedge e_2|}
\varphi(z'g)d^\PP z' \\ &=
\xi(\det g)\int_F\varphi((1,x)g)dx,
\end{aligned}
\]
where the last equality uses the affine chart $z'=(1,x)$, for which $z'\wedge e_2=1$. It follows that
\[
\begin{aligned}
\left(\iota_{\nu,\mu}(M_{\mu,\nu}\varphi)\right)(g)
&=
\nu(\det g)|\det g|^{1/2}
(M_{\mu,\nu}\varphi)(e_2g)\\
&=
\mu(\det g)|\det g|^{1/2}
\int_F\varphi((1,x)g)dx.
\end{aligned}
\]
On the other hand, since $\det(\dot w_0^{-1}u(x))=1$ and
\[
e_2\dot w_0^{-1}u(x)=(1,x),
\]
we have
\[
\begin{aligned}
\left(J_{\mu,\nu}f_\varphi\right)(g)
&=\int_F f_\varphi(\dot w_0^{-1}u(x)g)dx\\
&=\mu(\det g)|\det g|^{1/2}
\int_F\varphi(e_2\dot w_0^{-1}u(x)g)dx\\
&=
\mu(\det g)|\det g|^{1/2}
\int_F\varphi((1,x)g)dx.
\end{aligned}
\]
Therefore $\iota_{\nu,\mu}(M_{\mu,\nu}\varphi)
=
J_{\mu,\nu}(\iota_{\mu,\nu}\varphi)$, which proves the commutativity of the lower square.
\end{proof}

Note that under the Whittaker normalization in Section~\ref{sec:representations}, we have
\[
\xi(-1)R_{\mu,\nu}
=
\gamma(1,\xi^{-1},\psi)^{-1}J_{\mu,\nu}.
\]

We next come to the triple product setting. For $1\leq i\leq 3$, we simplify the notations as
\[
(\pi_i,\cH_i):=(\pi_{\mu_i,\nu_i},\cH_{\mu_i,\nu_i}),\qquad q_i:=q_{\mu_i,\nu_i}:\cS(F^2)\longrightarrow \cH_i,
\]
and denote by $\omega_i:=\mu_i\nu_i$ the central character of $\pi_i$. Again, we use the notations $$\bm{\chi}^{\#}:=(\mu_1,\nu_1;\mu_2,\nu_2;\mu_3,\nu_3),\qquad \xi_i:=\mu_i \nu_i^{-1}.$$ From now on, we assume that all $\mu_i,\nu_i$ are unitary and $\omega_1\omega_2\omega_3=\mathbf{1}$, so that
\[
\dim\Hom_{G^\Delta}(\pi_1\boxtimes \pi_2\boxtimes \pi_3,\mathbf{1})=1.
\]
We next put
\begin{equation}\label{eq:trilinear-kernel-characters}
\alpha_1:=\nu_1\mu_2\mu_3|\cdot|^{-1/2},\qquad \alpha_2:=\mu_1\nu_2\mu_3|\cdot|^{-1/2},\qquad \alpha_3:=\mu_1\mu_2\nu_3|\cdot|^{-1/2}.
\end{equation}
The assumption $\omega_1\omega_2\omega_3=\mathbf{1}$ gives
\begin{equation}\label{eq:kernel-character-identities}
\alpha_2\alpha_3=\xi_1|\cdot|^{-1},\qquad \alpha_3\alpha_1=\xi_2|\cdot|^{-1},\qquad \alpha_1\alpha_2=\xi_3|\cdot|^{-1}.
\end{equation}

Let
\[
\Omega:=\left\{(z_1,z_2,z_3)\in(F^2)^3\mid (z_2\wedge z_3)(z_3\wedge z_1)(z_1\wedge z_2)\neq0\right\},
\]
which is an open subset of $F^6$. We define a smooth function
\[
K_{\bm{\chi}^{\#}}:\Omega\longrightarrow \CC,\qquad K_{\bm{\chi}^{\#}}(z_1,z_2,z_3):=\alpha_1(z_2\wedge z_3)\alpha_2(z_3\wedge z_1)\alpha_3(z_1\wedge z_2),\qquad z_i\in F^2
\]
and regard it as a function on $F^6$ via extension by zero. The local integrability range of $|x|^s$ at $x=0$ is $\Re(s)>-1$, and our current setting is $\alpha_i = \alpha_i^\circ |\cdot|^{-1/2}$ with $\alpha_i^{\circ}$ unitary, so we can easily find that $K_{\bm{\chi}^{\#}}$ for our current $\bm{\chi}^{\#}$ is locally integrable (also see \cite[\S 10.2]{venkatesh-wang}). We then regard $K_{\bm{\chi}^{\#}}$ as a distribution on $\cS(F^6)$ via 
\[ \cS(F^6)\ni \Phi \longmapsto \langle K_{\bm{\chi}^{\#}},\Phi\rangle :=\int_{F^6}K_{\bm{\chi}^{\#}}(z)\Phi(z)dz.\]

For $\varphi_i\in\cH_i$, the function
\[
(z_1,z_2,z_3)\longmapsto K_{\bm{\chi}^{\#}}(z_1,z_2,z_3)\varphi_1(z_1)\varphi_2(z_2)\varphi_3(z_3)
\]
is $(-2)$-homogeneous in each variable separately. We therefore define
\[
L_{\bm{\chi}^{\#}}(\varphi_1,\varphi_2,\varphi_3):=
\int_{\PP^1(F)^3}K_{\bm{\chi}^{\#}}(z_1,z_2,z_3)\varphi_1(z_1)\varphi_2(z_2)\varphi_3(z_3)
\,d^\PP z_1\,d^\PP z_2\,d^\PP z_3,
\]
which also converges absolutely.

\begin{lemma}\label{lemma:homogeneous-trilinear-form}
The above distribution $\Phi\longmapsto \langle K_{\bm{\chi}^{\#}},\Phi\rangle$ descends to a nonzero $G^\Delta$-invariant trilinear form $L_{\bm{\chi}^{\#}}$ on $\cH_1\otimes \cH_2\otimes \cH_3$ along the quotient map
\[
q_{\bm{\chi}^{\#}}=q_1\otimes q_2\otimes q_3:\cS(F^2)\otimes\cS(F^2)\otimes\cS(F^2)\longrightarrow\cH_1\otimes\cH_2\otimes\cH_3.
\]
More precisely, for $\Phi_i\in\cS(F^2)$, we have
\[
\langle K_{\bm{\chi}^{\#}},\Phi_1\otimes\Phi_2\otimes\Phi_3\rangle
=
L_{\bm{\chi}^{\#}}\bigl(q_1(\Phi_1),q_2(\Phi_2),q_3(\Phi_3)\bigr).
\]
\end{lemma}

\begin{proof}
The identities \eqref{eq:kernel-character-identities} give
\[
K_{\bm{\chi}^{\#}}(t_1z_1,t_2z_2,t_3z_3)
=
\prod_{i=1}^3|t_i|^{-1}\xi_i(t_i)\,
K_{\bm{\chi}^{\#}}(z_1,z_2,z_3).
\]
Applying the projective integration formula successively in the three variables gives
\[
\langle K_{\bm{\chi}^{\#}},\Phi_1\otimes\Phi_2\otimes\Phi_3\rangle
=
L_{\bm{\chi}^{\#}}\bigl(q_1(\Phi_1),q_2(\Phi_2),q_3(\Phi_3)\bigr).
\]
The $G^\Delta$-invariance follows directly from the transformation property of the kernel together with the definition of the $G$-actions on the homogeneous models. The nonvanishing can be seen by choosing the three functions with sufficiently small supports contained in the open set on which all three wedge products are nonzero.
\end{proof}

We define the Fourier transform $\cF_{\psi}^{\otimes3}(K_{\bm{\chi}^{\#}})$ of the distribution $K_{\bm{\chi}^{\#}}$ by
\[
\langle\cF_{\psi}^{\otimes3}(K_{\bm{\chi}^{\#}}),\Phi\rangle
:=
\langle K_{\bm{\chi}^{\#}},\cF_{\psi^{-1}}^{\otimes3}\Phi\rangle,\qquad \Phi\in\cS(F^6).
\]
With this convention, the Fourier transform of a distribution represented by a locally integrable function agrees with the Fourier transform of that function defined above. The following is a reformulation of the distributional Fourier-transform identity  of \cite[Lemma 11.4.2]{venkatesh-wang}, rewritten in our notations.

\begin{theorem}[Venkatesh--Wang]\label{theorem:venkatesh-wang-fourier-identity}
In the present unitary range, we have
\[
\cF_\psi^{\otimes3}(K_{\bm{\chi}^{\#}})
=c(\bm{\chi}^{\#})K_{w_{123}\bm{\chi}^{\#}}
\]
as distributions on $\cS(F^6)$, with
\[
c(\bm{\chi}^{\#})=\left(\gamma(1,\alpha_1\alpha_2\alpha_3|\cdot|,\psi)\prod_{i=1}^3\gamma(1,\alpha_i,\psi)\right)^{-1}.
\]
In other words, we have the following commutative diagram.
\[
\begin{tikzcd}
\cS(F^6)\arrow[rr,"\cF_{\psi^{-1}}^{\otimes 3}"] \arrow[rd,"{c(\bm{\chi}^{\#})K_{w_{123}\bm{\chi}^{\#}}}"'] && \cS(F^6) \arrow[ld, "K_{\bm{\chi}^{\#}}"] \\
    & \CC &
\end{tikzcd}
\]
\end{theorem}

Note that here we have $|\alpha_i^{-1}(x)|/|x|=|x|^{-1/2}$, so the kernels on both sides lie in the same locally integrable range.

\begin{proof}
By Lemma 11.4.2 of \cite{venkatesh-wang}, we have
\[
\cF_{\psi}^{\otimes3}(K_{\bm{\chi}^{\#}})(z_1,z_2,z_3)
=
c(\bm{\chi}^{\#})
\frac{\alpha_1^{-1}(z_2\wedge z_3)\alpha_2^{-1}(z_3\wedge z_1)\alpha_3^{-1}(z_1\wedge z_2)}
{|(z_2\wedge z_3)(z_3\wedge z_1)(z_1\wedge z_2)|}
\]
as distributions. Let $\alpha_i^{w_{123}}$ denote the quasi-character obtained from the inducing data $w_{123}\bm{\chi}^{\#}$. The assumption $\omega_1\omega_2\omega_3=\mathbf{1}$ gives
\[
\alpha_i^{w_{123}}=\alpha_i^{-1}|\cdot|^{-1},\qquad 1\leq i\leq3.
\]
Hence the expression on the right-hand side is precisely $c(\bm{\chi}^{\#})K_{w_{123}\bm{\chi}^{\#}}$, which proves the claim.
\end{proof}

Now we turn to the case in which we are interested, namely the order 2, type II case for unitary principal series. By Corollary \ref{corollary:order-2-canonical-form}, we can take three mutually distinct characters $\xi_1,\xi_2,\xi_3:F^\times\longrightarrow\CC^\times$ of order 2 and $\xi_1\xi_2\xi_3=\mathbf{1}$, together with a unitary character $\lambda:F^\times\longrightarrow\CC^\times$ satisfying $\lambda^2=\mathbf{1}$ (not necessarily nontrivial), such that our L-packet is restricted from 
\[\bm{\pi}^{\#}:=(\xi_1\times \mathbf{1})\boxtimes (\xi_2\times \mathbf{1})\boxtimes (\lambda\xi_3\times \lambda).\]
The related characters are 
\[
\bm{\chi}^{\#}:=(\xi_1,\mathbf{1};\xi_2,\mathbf{1};\lambda\xi_3,\lambda),\qquad
w_{123}\bm{\chi}^{\#}=(\mathbf{1},\xi_1;\mathbf{1},\xi_2;\lambda,\lambda\xi_3).
\]

\begin{lemma}\label{lemma:type-ii-fourier-constant}
    In the current case, we have $K_{w_{123}\bm{\chi}^{\#}}=K_{\bm{\chi}^{\#}}$ as functions on $\Omega$, and hence as distributions on $\cS(F^6)$. Moreover, the constant in the Venkatesh--Wang formula becomes
\begin{equation}\label{eq:type-ii-fourier-constant}
c(\bm{\chi}^{\#})
=\left(
\prod_{\eta\in\{\mathbf{1},\xi_1,\xi_2,\xi_3\}}
\gamma(\frac12,\lambda\eta,\psi)
\right)^{-1}.
\end{equation}

\end{lemma}
\begin{proof}
    Direct computation. The definition \eqref{eq:trilinear-kernel-characters} of $\alpha_i$ now gives
    \[\alpha_1 = \lambda\xi_2\xi_3|\cdot|^{-1/2},\qquad \alpha_2 = \lambda\xi_1\xi_3|\cdot|^{-1/2},\qquad \alpha_3 = \lambda\xi_1\xi_2|\cdot|^{-1/2}\]
    for the character $\bm{\chi}^{\#}$. The three corresponding characters  $\alpha_i^{w_{123}}$ for $w_{123}\bm{\chi}^{\#}$ then become
    \[
    \alpha_1^{w_{123}}=\lambda\xi_1|\cdot|^{-1/2},\qquad
    \alpha_2^{w_{123}}=\lambda\xi_2|\cdot|^{-1/2},\qquad
    \alpha_3^{w_{123}}=\lambda\xi_3|\cdot|^{-1/2}.
    \]
    By assumption $\xi_1\xi_2\xi_3=\mathbf{1}$, we have $\alpha_i^{w_{123}}=\alpha_i$ for $1\leq i\leq 3$, hence $K_{w_{123}\bm{\chi}^{\#}}=K_{\bm{\chi}^{\#}}$. The formula \eqref{eq:type-ii-fourier-constant} is then a direct computation.
\end{proof}

We now pass from the preceding identity of distributions to the standard induced model. Put
\[
(\cH_{\bm{\chi}^{\#}},q_{\bm{\chi}^{\#}}):=\bigotimes_{i=1}^3 (\cH_{\mu_i,\nu_i},q_{\mu_i,\nu_i}),\qquad
(\cH_{w_{123}\bm{\chi}^{\#}},q_{w_{123}\bm{\chi}^{\#}}):=\bigotimes_{i=1}^3 (\cH_{\nu_i,\mu_i},q_{\nu_i,\mu_i}),
\]
\[
\bm{\chi}:=\bm{\chi}^{\#}|_{T_M}=w_{123}\bm{\chi}^{\#}|_{T_M}.
\]
 Restricting the tensor products of the isomorphisms $\iota_{\mu_i,\nu_i}$ and $\iota_{\nu_i,\mu_i}$ to $M$, we obtain isomorphisms
\[
\iota_{\bm{\chi}^{\#}}:\cH_{\bm{\chi}^{\#}}\overset{\sim}{\longrightarrow}I(\bm{\chi}),\qquad
\iota_{w_{123}\bm{\chi}^{\#}}:\cH_{w_{123}\bm{\chi}^{\#}}\overset{\sim}{\longrightarrow}I(w_{123}\bm{\chi})=I(\bm{\chi}).
\]

\begin{lemma}\label{lemma:type-ii-common-homogeneous-model}
In the current case, we have set-theoretic equalities
\[
\cH_{\bm{\chi}^{\#}}=\cH_{w_{123}\bm{\chi}^{\#}},\qquad
q_{\bm{\chi}^{\#}}=q_{w_{123}\bm{\chi}^{\#}},
\]
and $\iota_{\bm{\chi}^{\#}}=\iota_{w_{123}\bm{\chi}^{\#}}$ after identifying the two targets with $I(\bm{\chi})$. Consequently, $L_{\bm{\chi}^{\#}}=L_{w_{123}\bm{\chi}^{\#}}$ as linear forms on the common homogeneous model.
\end{lemma}

\begin{proof}
For each $1\leq i\leq3$, we have $\mu_i\nu_i^{-1}=\xi_i$ and $\xi_i^{-1}=\xi_i$. Hence both $\cH_{\mu_i,\nu_i}$ and $\cH_{\nu_i,\mu_i}$ consist of the smooth functions $\varphi:F^2\setminus\{0\}\longrightarrow\CC$ satisfying
\[
\varphi(tz)=|t|^{-1}\xi_i(t)\varphi(z),
\]
and
\[
q_{\mu_i,\nu_i}(\Phi)(z)
=\int_{F^\times}\Phi(tz)\xi_i(t)|t|\,d^\times t
=q_{\nu_i,\mu_i}(\Phi)(z).
\]
This verifies the set-theoretic equalities.

Let $m=(g_1,g_2,g_3)\in M$ and write $\det m:=\det g_i$. On a pure tensor in the common homogeneous model, the ratio between the actions attached to $w_{123}\bm{\chi}^{\#}$ and $\bm{\chi}^{\#}$ is
\[
\prod_{i=1}^3(\nu_i\mu_i^{-1})(\det m)=(\xi_1\xi_2\xi_3)^{-1}(\det m)=1.
\]
Thus the identity map is $M$-equivariant. The same computation (based on the definition of the maps $\iota_{\mu_i,\nu_i}$) gives
\[
\iota_{w_{123}\bm{\chi}^{\#}}(\varphi)(m)
=\prod_{i=1}^3(\nu_i\mu_i^{-1})(\det m)\,\iota_{\bm{\chi}^{\#}}(\varphi)(m)
=\iota_{\bm{\chi}^{\#}}(\varphi)(m).
\]
Finally, the preceding lemma gives $K_{w_{123}\bm{\chi}^{\#}}=K_{\bm{\chi}^{\#}}$, hence $L_{w_{123}\bm{\chi}^{\#}}=L_{\bm{\chi}^{\#}}$.
\end{proof}

We henceforth write
\[
\cH:=\cH_{\bm{\chi}^{\#}}=\cH_{w_{123}\bm{\chi}^{\#}},\qquad
q:=q_{\bm{\chi}^{\#}}=q_{w_{123}\bm{\chi}^{\#}},\qquad
L^{\mathrm{hom}}:=L_{\bm{\chi}^{\#}}=L_{w_{123}\bm{\chi}^{\#}},
\]
and
\[
\iota:=\iota_{\bm{\chi}^{\#}}=\iota_{w_{123}\bm{\chi}^{\#}}:\cH\overset{\sim}{\longrightarrow}I(\bm{\chi}).
\]
We transport the trilinear form to the standard induced model by setting
\[
L:=L^{\mathrm{hom}}\circ\iota^{-1}\in\Hom_{G^\Delta}(I(\bm{\chi}),\mathbf{1}).
\]

For $1\leq i\leq3$, applying Lemma~\ref{lemma:rank-one-fourier-intertwining}  with $(\mu,\nu,\psi)$ replaced by $(\nu_i,\mu_i,\psi^{-1})$, we get
\[
q_{\mu_i,\nu_i}\circ\cF_{\psi^{-1}}
=\gamma(1,\xi_i,\psi^{-1})^{-1}M_{\nu_i,\mu_i}\circ q_{\nu_i,\mu_i}.
\]
We therefore define
\[
R^{\mathrm{hom}}:=\bigotimes_{i=1}^3\left(\gamma(1,\xi_i,\psi^{-1})^{-1}M_{\nu_i,\mu_i}\right):\cH_{w_{123}\bm{\chi}^{\#}}\longrightarrow\cH_{\bm{\chi}^{\#}}.
\]
We can now summarize our discussion so far in the following commutative diagram:
\[
\begin{tikzcd}[column sep=huge,row sep=large]
\cS(F^6)
\arrow[d,"{q_{w_{123}\bm{\chi}^{\#}}}"']
\arrow[rrrr,"\cF_{\psi^{-1}}^{\otimes3}"]
&&&&
\cS(F^6)
\arrow[d,"q_{\bm{\chi}^{\#}}"]\\
\cH_{w_{123}\bm{\chi}^{\#}}
\arrow[rrrr,"{R^{\mathrm{hom}}}"]
\arrow[rrd,"{c(\bm{\chi}^{\#})L_{w_{123}\bm{\chi}^{\#}}}"']
&&&&
\cH_{\bm{\chi}^{\#}}
\arrow[lld,"L_{\bm{\chi}^{\#}}"]\\
&&\CC&&
\end{tikzcd}
\]
with
\[
\cF_{\psi^{-1}}(\Phi)(z):=\int_{F^2}\Phi(w)\psi^{-1}(z\wedge w)dw.
\]
(Note that $\psi^{-1}(x) = \psi(-x)$ and the self-dual measure $dx$ for $\psi^{-1}$ is the same as that for $\psi$.) The upper square is the tensor product of the three preceding identities, whereas the triangle follows from Theorem~\ref{theorem:venkatesh-wang-fourier-identity} in the form
\[
\langle K_{\bm{\chi}^{\#}},\cF_{\psi^{-1}}^{\otimes3}\Phi\rangle
=c(\bm{\chi}^{\#})\langle K_{w_{123}\bm{\chi}^{\#}},\Phi\rangle.
\]
Using the preceding Lemma~\ref{lemma:type-ii-common-homogeneous-model}, the commutativity of the triangle then becomes
\begin{equation}\label{eq:homogeneous-trilinear-eigenvalue}
L^{\mathrm{hom}}\circ R^{\mathrm{hom}}=c(\bm{\chi}^{\#})L^{\mathrm{hom}}.
\end{equation}

The final remaining task is to identify $R^{\mathrm{hom}}$ with the normalized self-intertwining operator on the standard induced model. Let $f\in V_{\nu_i,\mu_i}$ and set $\phi(y):=f(\dot w_0^{-1}u(y))$. For $x\neq0$, we have
\[
\dot w_0^{-1}u(x)\dot w_0^{-1}u(y)=
\begin{bmatrix}
 x^{-1}&-1\\
 &x
\end{bmatrix}
 \dot w_0^{-1}u(y-x^{-1}).
\]
Then a change of variables $t=x^{-1}$ gives
\[
(J_{\nu_i,\mu_i}f)(\dot w_0^{-1}u(y))
=\int_F\xi_i(t)|t|^{-1}\phi(y-t)dt.
\]
(See \S~\ref{subsec:restriction-compatibility} for the definitions of $J_{\nu_i,\mu_i},\ell_{\mu_i,\nu_i}^{\psi}$.) Consequently,
\[
\ell_{\mu_i,\nu_i}^{\psi}(J_{\nu_i,\mu_i}f)
=\left(\int_F\xi_i(t)|t|^{-1}\psi(-t)dt\right)
\ell_{\nu_i,\mu_i}^{\psi}(f)=\gamma(1,\xi_i,\psi^{-1})\ell_{\nu_i,\mu_i}^{\psi}(f).
\]
Here as usual in the context related to intertwining operators, the two integrals above are regarded as the regular value at $s=0$ of the meromorphic families
\[
\int_F\xi_i(t)|t|^{s-1}\phi(y-t)dt,\qquad \int_F\xi_i(t)|t|^{s-1}\psi(-t)\,dt\qquad (s\in \CC)
\]
of zeta integrals. Thus
\[
R_{\nu_i,\mu_i}:=\gamma(1,\xi_i,\psi^{-1})^{-1}J_{\nu_i,\mu_i}:V_{\nu_i,\mu_i}\longrightarrow V_{\mu_i,\nu_i}
\]
is really the Whittaker-normalized standard intertwining operator. By the comparison of $M_{\nu_i,\mu_i}, J_{\nu_i,\mu_i}$ in Lemma~\ref{lemma:rank-one-fourier-intertwining}, we get
\[
\iota_{\mu_i,\nu_i}\circ\gamma(1,\xi_i,\psi^{-1})^{-1}M_{\nu_i,\mu_i}
=R_{\nu_i,\mu_i}\circ\iota_{\nu_i,\mu_i}.
\]

Taking the tensor product over $1\leq i\leq3$ and restricting to $M$ gives the commutative diagram
\[
\begin{tikzcd}[column sep=huge]
\cH_{w_{123}\bm{\chi}^{\#}}
\arrow[d,"{\iota_{w_{123}\bm{\chi}^{\#}}}"']
\arrow[rrrr,"{R^{\mathrm{hom}}}"]
&&&&
\cH_{\bm{\chi}^{\#}}
\arrow[d,"\iota_{\bm{\chi}^{\#}}"]\\
I(w_{123}\bm{\chi})
\arrow[rrrr,"{R(w_{123},w_{123}\bm{\chi})}"]
&&&&
I(\bm{\chi})
\end{tikzcd}
\]
by the restriction compatibility Lemma~\ref{lemma:normalized-intertwining-restriction}.
In the current type~II case, $W_{\bm{\chi}}=\{1,w_{123}\}$ and $W_{\bm{\chi}}^0=\{1\}$. Let $r_{123}$ be the image of $w_{123}=w_1w_2w_3$ in $R_{\bm{\chi}}$, so that $r_{123}$ is the unique nontrivial element of $R_{\bm{\chi}}$. Since $w_{123}\bm{\chi}=\bm{\chi}$, the bottom arrow in the above diagram is the self-intertwining operator
\[
R(w_{123},w_{123}\bm{\chi})= R(r_{123},\bm{\chi}):I(\bm{\chi})\longrightarrow I(\bm{\chi}).
\]
Under the common identifications fixed above, the commutativity of the diagram reads
\begin{equation}\label{eq:homogeneous-whittaker-intertwining-comparison}
\iota\circ R^{\mathrm{hom}}=R(r_{123},\bm{\chi})\circ\iota.
\end{equation}

\begin{proposition}\label{proposition:type-ii-period-eigenvalue}
For the trilinear form $L=L^{\mathrm{hom}}\circ\iota^{-1}$ defined above, we have
\[
L\circ R(r_{123},\bm{\chi})=c(\bm{\chi}^{\#})L.
\]
\end{proposition}

\begin{proof}
Equations \eqref{eq:homogeneous-trilinear-eigenvalue} and \eqref{eq:homogeneous-whittaker-intertwining-comparison} give
\[
L\circ R(r_{123},\bm{\chi})\circ\iota
=L^{\mathrm{hom}}\circ R^{\mathrm{hom}}=c(\bm{\chi}^{\#})L^{\mathrm{hom}}=c(\bm{\chi}^{\#})L\circ\iota.
\]
Since $\iota$ is an isomorphism, the result follows.
\end{proof}

Now let us complete the proof of the order 2, type II case of Theorem \ref{theorem:order-2-main-theorem}.

\begin{corollary}\label{corollary:type-ii-distinguished-constituent}
Put (see Theorem \ref{theorem:order-2-main-theorem})
\[
Y:=\{\mathbf{1},\xi_1,\xi_2,\xi_3\},\qquad
s:=\prod_{\eta\in Y}\varepsilon(\frac12,\eta,\psi)\in \{\pm 1\}.
\]
Then we have
\[
L\circ R(r_{123},\bm\chi)=sL,\qquad L\in\Hom_{G^\Delta}(I(\bm\chi),\mathbf{1}),
\]
and consequently, 
\[
\dim\Hom_{G^\Delta}(\bm{\sigma}_s,\mathbf{1})=1,
\qquad
\dim\Hom_{G^\Delta}(\bm{\sigma}_{-s},\mathbf{1})=0.
\]
\end{corollary}

The only needed nontrivial fact is that we have the following equality:
\[ \prod_{\eta\in Y}\varepsilon(\frac12,\lambda\eta,\psi) = \prod_{\eta\in Y}\varepsilon(\frac12,\eta,\psi). \]
Indeed, for $a\in F^\times$ write $\eta_a(x):=(a,x)_F$, where $(\cdot,\cdot)_F$ is the quadratic Hilbert symbol, and put $w_\psi(a):=\varepsilon(\frac12,\eta_a,\psi)$. By \cite{szpruch}, $w_\psi(a)$ is the normalized Weil index for the character $\overline{\psi}$, then Theorem~A.4 and Corollary~A.5 of \cite{ranga-rao} give
\[
w_\psi(cd)=(c,d)_F w_\psi(c)w_\psi(d),
\qquad w_\psi(c)^4=1.
\]
We can choose $a,b,c\in F^\times$ such that $\xi_1=\eta_a$, $\xi_2=\eta_b$, $\xi_3=\eta_{ab}$ and $\lambda=\eta_c$. We write $D=\{1,a,b,ab\}$, then the ratio of the two products is given by
\[
\prod_{d\in D}\frac{w_\psi(cd)}{w_\psi(d)}
=w_\psi(c)^4\prod_{d\in D}(c,d)_F
=w_\psi(c)^4(c,a^2b^2)_F
=1.
\]
Now the remaining part of the proof is routine. 

\begin{proof}
Since $(\lambda\eta)^2=\mathbf{1}$ for $\eta\in Y$, the definition of the local gamma factor gives
\[
\gamma(\frac12,\lambda\eta,\psi)
=\varepsilon(\frac12,\lambda\eta,\psi).
\]
Thus from \eqref{eq:type-ii-fourier-constant} and the fact $s^2=1$ (see remarks under Theorem \ref{theorem:order-2-main-theorem}), we get 
\[
c(\bm\chi^{\#})
=\left(\prod_{\eta\in Y}\varepsilon(\frac12,\lambda\eta,\psi)\right)^{-1}
=s^{-1}=s.
\]
The first assertion then follows from Proposition~\ref{proposition:type-ii-period-eigenvalue}. For $t\in\{\pm1\}$ and $v\in\bm{\sigma}_t$, Section~\ref{sec:representations} has explained $R(r_{123},\bm\chi)v=t^{-1}v=tv$, so that $(t-s)L(v)=0$. Thus a nonzero $L$ vanishes on $\bm{\sigma}_{-s}$ and is then nonzero on $\bm{\sigma}_s$. 
\end{proof}

% ===== End 05_type 2.tex =====

% ===== Begin 06_references.tex =====

% ===== End 06_references.tex =====
\end{document}